\documentclass[3p,11pt]{elsarticle}
\usepackage{graphicx,subfig}
\usepackage{stmaryrd}
\usepackage{commath}
\usepackage{url}
\usepackage{amsmath}
\usepackage{amssymb}
\usepackage{amsthm}
\usepackage{mathrsfs}
\usepackage[numbers]{natbib}
\usepackage[hidelinks]{hyperref}
\hypersetup{
	colorlinks   = true, 
	urlcolor     = red, 
	linkcolor    = red, 
	citecolor   = red 
}
\usepackage{cleveref}
\usepackage{enumitem}

\DeclareFontFamily{U}{matha}{\hyphenchar\font45}
\DeclareFontShape{U}{matha}{m}{n}{
	<5> <6> <7> <8> <9> <10> gen * matha
	<10.95> matha10 <12> <14.4> <17.28> <20.74> <24.88> matha12
}{}
\DeclareSymbolFont{matha}{U}{matha}{m}{n}
\DeclareFontSubstitution{U}{matha}{m}{n}

\DeclareFontFamily{U}{mathx}{\hyphenchar\font45}
\DeclareFontShape{U}{mathx}{m}{n}{
	<5> <6> <7> <8> <9> <10>
	<10.95> <12> <14.4> <17.28> <20.74> <24.88>
	mathx10
}{}
\DeclareSymbolFont{mathx}{U}{mathx}{m}{n}
\DeclareFontSubstitution{U}{mathx}{m}{n}

\DeclareMathDelimiter{\vvvert}{0}{matha}{"7E}{mathx}{"17}

\newcommand{\tnorm}[1]{{\left\vvvert #1 \right\vvvert}}
\newcommand{\jump}[1]{\llbracket #1 \rrbracket}
\newcommand{\av}[1]{\{\!\!\{#1\}\!\!\}}

\makeatletter
\def\ps@pprintTitle{%
	\let\@oddhead\@empty
	\let\@evenhead\@empty
	\def\@oddfoot{}%
	\let\@evenfoot\@oddfoot}

\newtheorem{theorem}{Theorem}[section]
\newtheorem{lemma}{Lemma}[section]
\newtheorem{proposition}{Proposition}[section]

\theoremstyle{definition}
\newtheorem{definition}{Definition}[section]

\theoremstyle{remark}
\newtheorem{remark}{Remark}[section]

\begin{document}

\begin{frontmatter}	
	\title{Convergence to weak solutions of a space-time hybridized discontinuous Galerkin method for 
		the incompressible Navier--Stokes equations}
	
	\author[KLAK]{Keegan L. A. Kirk\fnref{label1}}
	\ead{k4kirk@uwaterloo.ca}
	\address[KLAK]{Department of Applied Mathematics, University of
		Waterloo, Canada} 
	
	\author[AC]{Ay\c{c}{\i}l \c{C}e\c{s}mel\.{i}o\u{g}lu\fnref{label2}}
	\ead{cesmelio@oakland.edu}
	\address[AC]{Department of Mathematics and Statistics, Oakland
		University, U.S.A.}
	
	\author[KLAK]{Sander Rhebergen\fnref{label3}}
	\ead{srheberg@uwaterloo.ca}
	
	\begin{abstract}
	We prove that a space-time hybridized discontinuous Galerkin method for the evolutionary Navier--Stokes equations converges
	to a weak solution as the time step and mesh size tend to zero. Moreover, we show that this weak solution satisfies the energy inequality.
	To perform our analysis, we make use of discrete functional analysis tools and a discrete version of the Aubin--Lions--Simon theorem.
	\end{abstract}

	\begin{keyword}
		Navier--Stokes \sep space-time \sep hybridized
		\sep discontinuous Galerkin \sep minimal regularity
	\end{keyword}
	
\end{frontmatter}

	\section{Introduction}
	\label{s:introduction}
	In this article, we continue our study of the space-time hybridized discontinuous Galerkin (HDG) method
	for the evolutionary incompressible Navier--Stokes equations
	analyzed in \cite{Kirk:2021}. Therein, we proved that the method is \emph{pressure-robust} (see \Cref{ss:related_results} below) and 
	derived optimal rates of convergence in space and time for the velocity
	field which are independent of the pressure.
	However, as a discontinuous method, the HDG method introduces additional
	stabilization which is a potential source of consistency error
	if the exact solution is not sufficiently regular. 
	Consequently,
	the convergence results deduced from the standard a priori analysis of DG and HDG methods often 
	exclude the case of non-smooth solutions which may be present
	in physically realistic scenarios.
	
	For this reason, our analysis in \cite{Kirk:2021} considered \emph{strong solutions}
	of the Navier--Stokes system, and cannot be used to deduce convergence
	to \emph{weak solutions} in the absence
	of additional regularity. 
	This restricts
	the applicability of
	our analysis in two important ways. The first is that in three spatial dimensions,
	the existence of a strong solution is guaranteed only for sufficiently small problem data (see e.g. \cite{Boyer:book,Constantin:book,Temam:book}).
	The second is that, for a polygonal ($d=2$) or polyhedral domain $(d=3)$, the exact solution is known to possess
	sufficient regularity for consistency only under further geometric assumptions on the domain (e.g. \emph{convexity}); see \cite{Dauge:1989}.
	
	The purpose of this paper is to fill this gap. To circumvent the problems
	posed by the lack of consistency in our numerical scheme, we
	instead consider the concept of \emph{asymptotic consistency} introduced
	in \cite[Section 5.2]{Pietro:book}.
	That is, we aim to show the discrete weak formulation resulting from our
	space-time HDG discretization converges to the exact weak formulation of the Navier--Stokes equations in
	a suitable sense
	as the mesh size tends to zero. 
	This is made challenging by the discontinuous nature of our numerical method,
	as standard compactness results like the Rellich--Kondrachov theorem \cite[Theorem III.2.34]{Boyer:book}
	and the Aubin--Lions--Simon theorem \cite[Theorem II.5.16]{Boyer:book} are lost and appropriate
	discrete analogues must be derived.
	
	Fortunately, discrete compactness for DG schemes is, at this point, well studied. We mention
	in particular the works of Buffa and Ortner \cite{Buffa:2009}, Di Pietro and Ern \cite{DiPietro:2010}, and
	Kikuchi \cite{Kikuchi:2012}, wherein discrete versions
	of the Rellich--Kondrachov theorem are proven for broken Sobolev and broken polynomial spaces. A common theme among these works is the introduction
	of a discrete analogue of the gradient operator that incorporates information from
	the \emph{jumps} of the discrete solution across its discontinuities. As
	for a discrete analogue of the Aubin--Lions--Simon theorem in the time-dependent setting, we mention
	the work of Walkington \cite{Walkington:2010} where it is shown
	that DG time stepping methods enjoy similar compactness
	properties to the evolutionary equations they are used to approximate.
	Unfortunately, the results of \cite{Walkington:2010} are valid only
	for conforming spatial discretizations. This was remedied in \cite{Li:2015},
	wherein a generalization of the work of Walkington valid
	for broken Sobolev spaces (and thus, for a broad class of non-conforming discretizations)
	is obtained.
	
	In this article, we adapt some of the available discrete functional analysis
	tools \cite{DiPietro:2010} to the HDG setting (see also \cite{Kikuchi:2012} for similar efforts).
	We also prove a variation of the discrete Aubin--Lions--Simon theorem
	in \cite{Walkington:2010} valid for our non-conforming discretization.
	Our result differs slightly from that of \cite{Li:2015} in that we stay entirely
	within the framework of broken polynomial spaces. In an effort to unify the available discrete
	functional analysis tools for spatial DG discretizations and DG time stepping, 
	we introduce a discrete time derivative operator in analogy with the aforementioned
	discrete gradient operator using the time lifting operator in \cite{Schotzau:2010}, and we show that some of the assumptions required in \cite{Li:2015,Walkington:2010}
	for compactness can be interpreted using this discrete time derivative. The main contributions of this work can be summarized as follows:
	\begin{enumerate}[label=(\roman*)]
		\item We prove that the discrete solution of space-time HDG scheme analyzed in \cite{Kirk:2021} for strong solutions converges to a Leray--Hopf weak solution of
		the evolutionary Navier--Stokes equations. To our knowledge, this is one of the few minimal regularity convergence results available
		for HDG discretizations in general, and it is the first for a space-time HDG discretization. As a byproduct, we obtain
		a new proof of the existence of weak solutions of the Navier--Stokes equations.
		\item We introduce the notion of a discrete time derivative and show that it serves
		as an appropriate approximation of the distributional time derivative.
		\item We obtain a minor variation of the discrete Aubin--Lions--Simon compactness theorem proven in \cite[Theorem 3.1]{Walkington:2010}
		valid for broken polynomial spaces, as a special case of \cite[Theorem 3.2]{Li:2015}, using the discrete functional analysis tools of DiPietro and Ern \cite{DiPietro:2010}.
		Moreover, we show that uniform bounds on the discrete time derivative in a discrete dual space suffice for compactness for DG time stepping schemes, by analogy to the corresponding continuous theorem.
	\end{enumerate}
	The article is organized as follows: in the remainder of \Cref{s:introduction}, we give a brief
	overview of the relevant literature and introduce some of the notation that will be used throughout.
	In \Cref{s:preliminaries}, we introduce the space-time HDG method under consideration
	and recall some of the key results obtained in \cite{Kirk:2021}. In \Cref{s:disc_operators},
	we introduce discrete analogues of the gradient operator and time derivative, and
	recast the numerical scheme in terms of these discrete operators. In \Cref{s:bnds_compact},
	we prove that these discrete operators are bounded uniformly with respect to the mesh size
	and time step, and as a consequence we obtain convergence of the sequence
	of discrete velocity solutions as the mesh size and time step tend to zero. In \Cref{s:convergence}, we show that the
	limit of this sequence of discrete solutions is a weak solution to the Navier--Stokes equations.

	\subsection{Related results}
	\label{ss:related_results}
	A shortfall of many popular discretizations for incompressible flows
	is that they fail to preserve a fundamental invariance property
	enjoyed at the continuous level: perturbations in the external forcing
	term by a gradient field should not influence the velocity field
	\cite{Linke:2014}. Violation of this invariance property at the
	discrete level can lead to a coupling of the approximation errors in
	the velocity and pressure and introduces an unfavourable scaling in
	the velocity error with respect to inverse powers of the kinematic
	viscosity.  Discretizations that do preserve this invariance property
	have been coined \emph{pressure-robust}, and have recently been the
	subject of intensive study
	\cite{Fu:2019,John:2017,Lederer:2017,Lederer:2017b,Lehrenfeld:2016,Linke:2014,Schroeder:2018}.
	
	The two requisites to ensure that a discretization is pressure-robust
	are exact satisfaction of the divergence constraint and $H(\text{div};\Omega)$-conformity
	of the velocity field. While the latter is automatically satisfied for
	classical conforming discretizations like the Taylor--Hood element \cite{Hood:1974}, the former is not.
	In fact, conforming discretizations with exact mass conservation are scarce; 
	examples include the Scott--Vogelius element \cite{Scott:1985} and the more
	recent Guzm\'an--Neilan element \cite{Guzman:2014b, Guzman:2014a}. Instead, one can use $H(\text{div};\Omega)$-conforming DG methods such as the Raviart--Thomas (RT) element 
	or the Brezzi--Douglas--Marini (BDM) element \cite{Boffi:book} to obtain pressure-robust discretizations.
	To our knowledge, the use of $H(\text{div};\Omega)$-conforming DG methods
	for incompressible flows was first explored in \cite{Cockburn:2007}.
	
	In \cite{Rhebergen:2018a}, a simple class of pressure-robust HDG
	methods is introduced.  By leveraging hybridization, an
	$H(\text{div};\Omega)$-conforming discrete velocity field can be
	obtained while working with completely discontinuous finite element
	spaces.  In \cite{Horvath:2019, Horvath:2020}, a space-time HDG method
	that remains pressure-robust on moving and deforming domains is
	introduced based on this discretization; however, a rigorous analysis
	of this method is lacking. Our work in \cite{Kirk:2021} is a first
	step in this direction. We mention also the papers
	\cite{Ahmed:2020,Chrysafinos:2010,Masud:1997,Rhebergen:2012,Rhebergen:2013b}
	wherein alternative space-time finite element discretizations of the
	Stokes and Navier--Stokes problem are considered.
	
	We remark that our analysis of convergence to weak solutions is
	greatly facilitated by the $H(\text{div};\Omega)$-conformity of the
	discrete velocity as well as the fact that it is point-wise solenoidal. This
	ensures that, though our finite element spaces are completely
	discontinuous, the discrete velocity solution belongs to the
	continuous space $H$ defined below. This is particularly important
	when we seek to extract a convergent subsequence of approximate
	velocities in \Cref{s:bnds_compact}, as no additional effort will be
	required to show that accumulation points satisfy the
	incompressibility constraint. Moreover, this guarantees that the
	discrete time derivative defined in \Cref{s:disc_operators} actually
	belongs to the same continuous dual space as the distributional time
	derivative of the exact solution.
	
	\subsection{Notation}
	We use standard notation for Lebesgue and Sobolev spaces: given a
	bounded measurable set $D$, we denote by $L^p(D)$ the space of
	$p$-integrable functions. When $p=2$, we denote the $L^2(D)$ inner
	product by $(\cdot,\cdot)_D$. We denote by $W^{k,p}(D)$ the Sobolev
	space of $p$-integrable functions whose distributional derivatives up to order $k$ are
	$p$-integrable.  When $p=2$, we write $W^{k,p}(D) = H^k(D)$. We
	define $H_0^1(D)$ to be the subspace of $H^1(D)$ of functions with
	vanishing trace on the boundary of $D$.  We denote the space of
	polynomials of degree $k\ge0$ on $D$ by $P_k(D)$. We use 
	standard notation for spaces of vector valued functions with $d$ components, e.g. $L^2(D)^d$, $H^k(D)^d$, $P^k(D)^d$, etc.
	At times we drop the superscript for convenience, e.g. we denote by $\norm{\cdot}_{L^2(\Omega)}$  
	the norm on both $L^2(\Omega)$ and $L^2(\Omega)^d$.
	
	Next, let $U$ be a Banach space, $I = [a,b]$ an interval in $\mathbb{R}$, and $1 \le p < \infty$.
	We denote by
	$L^p(I;U)$ the Bochner space of $p$-integrable functions defined on 
	$I$ taking values in $U$. 
	When $p = \infty$, we denote by $L^{\infty}(I;U)$ the Bochner space of 
	essentially bounded functions taking values in $U$ and by $C(I;U)$ the space
	of (time) continuous functions taking values in $U$. 
	Finally, we let
	$P_k(I;U)$ denote the space of 
	polynomials of degree $k \ge 0$ in time taking values in $U$.
	
	\section{Preliminaries}
	\label{s:preliminaries}
	In this section, we discuss the weak formulation for the continuous Navier--Stokes problem \cref{eq:ns_equations},
	introduce the space-time HDG method that we will use to approximate solutions of  \cref{eq:ns_equations}, and collect a number 
	of useful results for our analysis.
	
	\subsection{The continuous problem}
	Given a suitably chosen body force $f$, kinematic viscosity $\nu \in 
	\mathbb{R}^+$, and initial data $u_0$, we consider the transient Navier--Stokes 
	system posed on a bounded Lipschitz 
	domain
	$\Omega \subset \mathbb{R}^d$, $d \in \cbr{2,3}$:
	\begin{subequations} \label{eq:ns_equations} 
		\begin{align} \label{eq:ns_equations_a}
			\partial_t u - \nu \Delta u + \nabla \cdot( u \otimes u) + \nabla p &= f, 
			&& 
			\hspace{-10mm}
			\text{in } \Omega\times (0,T], \\
			\nabla \cdot u &= 0, && \hspace{-10mm}\text{in } \Omega\times (0,T], \\
			u &=0, && \hspace{-10mm}\text{on }  \partial \Omega\times (0,T], \\
			u(x,0) &= u_0(x), && \hspace{-10mm}\text{in } \Omega. \label{eq:ns_equations_d}
		\end{align}
	\end{subequations}
	To avoid any complications arising from curved boundaries, we will assume 
	further that 
	in two spatial dimensions $\Omega$ is a polygon and in three spatial dimensions $\Omega$ 
	is a polyhedron. As we are interested in weak solutions, 
	we require no assumption that $\Omega$ is convex.

	We begin our discussion of the Navier--Stokes system with the
	theory of weak solutions of Leray--Hopf type
	\cite{Boyer:book,Temam:book}. The starting point is
	$\mathscr{V} = \cbr{ u \in C_c^{\infty}(\Omega)^d \; | \; \nabla \cdot
		u = 0}$, the space of compactly supported solenoidal smooth vector fields.  We define
	two function spaces, $H$ and $V$, as the closures of $\mathscr{V}$ in
	the norm topologies of $L^2(\Omega)$ and $H_0^1(\Omega)$,
	respectively. For an open, bounded Lipschitz set $\Omega$, we have the
	following characterizations of $H$ and $V$ \cite[Theorems I.1.4 and
	I.1.6]{Temam:book}:
	\begin{equation}
		\label{eq:continuous_spaces}
		H = \cbr[1]{u \in L^2(\Omega)^d \; | \; \nabla \cdot u = 0 \text{ and } 
			u \cdot n|_{\partial \Omega} = 0 },
		\
		V = \cbr[1]{u \in H_0^1(\Omega)^d \; | \; \nabla \cdot u = 0}.
	\end{equation}
	Note that
	$H \subset H(\text{div};\Omega):=\cbr{v \in L^2(\Omega)^d\,|\, \nabla
		\cdot u \in L^2(\Omega)}$. We equip $H$ and $V$ with the standard
	norms on $L^2(\Omega)^d$ and $H_0^1(\Omega)^d$, respectively.
	\begin{definition}[Weak solution]\label{def:equiv_weak_form}
		Given a body force $f \in L^2(0,T;H^{-1}(\Omega)^d)$ and an initial condition $u_0 \in H$, a function $u \in 
		L^{\infty}(0,T;H) \cap L^2(0,T;V)$ with $\od{u}{t} \in L^1(0,T;V')$ is
		said to be a weak solution of the Navier--Stokes equations
		\cref{eq:ns_equations} provided it satisfies for all
		$\varphi \in C_c(0,T;V)$, 
		\begin{equation} \label{eq:ns_weak_formulation2}
			\int_0^T \bigg \langle \od{u}{t}, \varphi \bigg\rangle_{V' \times V} 
			\dif t 
			+ 
			\int_0^T((u \cdot \nabla)u , \varphi) \dif t + \nu \int_0^T (\nabla u, 
			\nabla 
			\varphi) \dif t = \int_0^T \langle f, \varphi \rangle_{H^{-1} \times H_0^1} 
			\dif t,
		\end{equation}
		and $u(0) = u_0$ in $V'$ (see e.g. \cite[Section V.1.2.2]{Boyer:book}). 
	\end{definition}
	It is well known that weak solutions in the sense of \Cref{def:equiv_weak_form}
	are weakly continuous from $[0,T]$ into $H$, and their distributional
	time derivative possess
	the further regularity $\od{u}{t} \in L^{4/d}(0,T;V')$ (see e.g. \cite{Boyer:book}).
	If $d=2$, this solution is unique and 
	furthermore  $u \in C(0,T;H)$. Uniqueness in three dimensions remains an 
	open problem.
	\begin{remark}[On the regularity of the body force]
		Our main result (\Cref{thm:convergence}) should remain valid for $f \in L^2(0,T;H^{-1}(\Omega)^d)$
		provided there is an appropriate smoothing operator $E_h: V_h \to H_0^1(\Omega)^d$; see e.g. \cite{Badia:2014}.
		In particular, if $v_h \to v$ strongly in $L^2(\Omega)^d$, we require also that $E_h v_h \to v$ strongly in $L^2(\Omega)^d$ as $h \to 0$.
		For simplicity, we focus on body forces $f \in L^2(0,T;L^2(\Omega)^d)$.
	\end{remark}
	\begin{remark}[The energy inequality] \label{rem:energy_ineq}
		In two dimensions, the weak solution to the Navier--Stokes equations
		satisfies the following \emph{energy equality}: for all $s \in (0,T)$,
		\begin{equation} \label{eq:energy_eq}
			\norm{u(s)}_{L^2(\Omega)}^2 + 2\nu \int_0^s \norm{u}_V^2 \dif t \; = \;	\norm{u_0}_{L^2(\Omega)}^2 + 2 \int_0^s \langle f, u \rangle_{H^{-1} \times H_0^1}  \dif t.
		\end{equation}
		In three dimensions, we say that a weak solution is of \emph{Leray--Hopf}
		type if it satisfies the \emph{energy inequality}: for a.e. $s \in (0,T)$,
		\begin{equation}\label{eq:energy_ineq}
			\norm{u(s)}_{L^2(\Omega)}^2 + 2\nu \int_0^s \norm{u}_V^2 \dif t \; \le \;	\norm{u_0}_{L^2(\Omega)}^2 + 2 \int_0^s \langle f, u \rangle_{H^{-1} \times H_0^1}  \dif t.
		\end{equation}
	\end{remark}

	\subsection{Space-time setting and finite element spaces}
	\label{ss:stdesc}
	In this subsection, we will introduce the space-time slabs, elements,
	faces, and finite element spaces required for the space-time HDG
	discretization. We follow some of the definitions introduced in
	\cite{Pietro:book2}.  We define a simplicial mesh of $\Omega$ to be a
	couple $(\mathcal{T}_h,\mathcal{F}_h)$ where the set of mesh elements
	$\mathcal{T}_h$ is a finite collection of nonempty, disjoint simplices
	$K$ with boundary $\partial K$ and diameter $h_K$ such that
	$\overline{\Omega} = \bigcup_{K \in \mathcal{T}_h} \overline{K}.$ We
	define the \emph{mesh size} $h$ of $\mathcal{T}_h$ to be
	$h = \max_{K \in \mathcal{T}_h} h_K.$
	
	The set of mesh faces $\mathcal{F}_h$ is a finite collection of
	nonempty, disjoint subsets of $\overline{\Omega}$ such that, for any
	$F \in \mathcal{F}_h$, $F$ is a non-empty, connected subset of a
	hyperplane in $\mathbb{R}^d$.  We assume further that for each
	$F \in \mathcal{F}_h$, either there exist distinct mesh elements
	$K_1, K_2 \in \mathcal{T}_h$ such that
	$F = \partial K_1 \cap \partial K_2$, in which case we call $F$ an
	interior face, or there exists one mesh element $K \in \mathcal{T}_h$
	such that $F = \partial K \cap \partial \Omega$ and we call $F$ a
	boundary face. Moreover, we assume that the set of mesh faces forms a
	partition of the mesh skeleton; that is,
	$ \partial \mathcal{T}_h = \bigcup_{K \in \mathcal{T}_h} \partial K =
	\bigcup_{F \in \mathcal{F}_h} F$. We collect interior faces
	in the set $\mathcal{F}_h^i$ and boundary faces in the set
	$\mathcal{F}_h^b$. Note that
	$\mathcal{F}_h = \mathcal{F}_h^i \cup \mathcal{F}_h^b$.
	
	Let $k_s \ge 1$ be a fixed integer. We introduce a pair of
	discontinuous finite element spaces on $\mathcal{T}_h$:
	\begin{subequations} \label{eq:fe_space_element}
		\begin{align}
			V_h &:= \cbr[1]{v_h \in L^2(\Omega)^d \; | \; v_h \in \label{eq:vel_space_element}
				P_{k_s}(K)^{d} \; \forall K \in \mathcal{T}_h}, \\
			Q_h &:= \cbr[1]{q_h \in L_0^2(\Omega) \; | \; q_h \in 
				P_{k_s-1}(K) \; \forall K \in \mathcal{T}_h}, \label{eq:pres_space_element}
		\end{align} 
	\end{subequations}
	and on $\partial \mathcal{T}_h$, we introduce a pair of discontinuous facet finite element 
	spaces:
	\begin{equation*}
		\begin{split}
			\bar{V}_h &:= \cbr[1]{\bar v_h \in L^2(\partial \mathcal{T}_h) 
				\; | \; 
				\bar v_h\in 
				P_{k_s}(F)^{d} \; \forall F \in \mathcal{F}_h, \; \bar{v}_h|_{\partial \Omega} = 0 }, \\
			\bar{Q}_h &:= \cbr[1]{\bar q_h \in L^2(\partial \mathcal{T}_h) \; | \; \bar 
				q_h \in 
				P_{k_s}(F) \; \forall F \in \mathcal{F}_h}.
		\end{split}
	\end{equation*}

	Next, we partition the time interval $(0,T)$ into a series of $N+1$
	time-levels $0 = t_0 < t_1 < \dots < t_N = T$ of length
	$\Delta t_n = t_{n+1} - t_n$, and we define
	$\tau = \max_{0 \le n \le N-1} \Delta t_n$. We assume this time
	partition is quasi-uniform, i.e. there exists a $C_{U'}>0$ such that
	$\tau \le C_{U'} \Delta t_n$ for all $n=0,\dots,N-1$.  A space-time
	slab is then defined as $\mathcal{E}^n = \Omega \times I_n$, with
	$I_n = (t_n,t_{n+1})$.  We tessellate the space-time slab
	$\mathcal{E}^n$ with space-time prisms $\mathscr{K} = K\times I_n$.
	We denote this tessellation by $\mathfrak{T}_h^n$ and note that
	$\mathcal{E}^n = \bigcup_{\mathscr{K} \in \mathfrak{T}_h^n}
	\mathscr{K}$.  Combining each space-time slab $n=0,\dots,N-1$, we
	obtain a tessellation of the space-time domain
	$\mathfrak{T}_h = \bigcup_{n = 0}^{N-1} \mathfrak{T}_h^n$.
	
	We next discuss space-time faces. Consider a single space-time slab
	$\mathcal{E}^n$. We define a space-time face
	$\mathscr{F} = F\times I_n$ to be an interior face if
	$F \in \mathcal{F}_{h}^i$, and a boundary face if
	$F \in \mathcal{F}_{h}^b$.  We collect the set of space-time
	interior faces in the set $\mathfrak{F}_{h,n}^i$ and the set of
	space-time boundary faces in the set $\mathfrak{F}_{h,n}^b$. We
	collect all space-time faces in the space-time slab $\mathcal{E}^n$ in
	the set $\mathfrak{F}_{h,n}$. Combining each space-time slab
	$n=0,\dots, N-1$, we obtain the set of all space-time interior faces
	$\mathfrak{F}_{h}^i = \bigcup_{n=0}^{N-1} \mathfrak{F}^i_{h,n}$, the
	set of all space-time boundary faces
	$\mathfrak{F}_{h}^b = \bigcup_{n=0}^{N-1} \mathfrak{F}^b_{h,n}$, and
	the set of all space-time faces
	$\mathfrak{F}_{h} = \mathfrak{F}_h^i \cup \mathfrak{F}_h^b$. We assume
	that the set of mesh faces forms a partition of the space-time mesh
	skeleton; that is,
	$\partial \mathfrak{T}_h = \bigcup_{n=0}^{N-1} \bigcup_{K \in
		\mathcal{T}_h} \partial K \times I_n = \bigcup_{\mathscr{F} \in
		\mathfrak{F}_h} \mathscr{F}$.
	
	We define a space-time prismatic mesh of $\Omega \times (0,T)$ to be
	the couple $(\mathfrak{T}_h,\mathfrak{F}_h)$. We will say that a
	space-time prismatic mesh $(\mathfrak{T}_h,\mathfrak{F}_h)$ is
	\emph{shape-regular} if the underlying spatial mesh
	$(\mathcal{T}_h,\mathcal{F}_h)$ is shape-regular; i.e. there exists a
	constant $C_R > 0$ such that $h_K/\rho_K \le C_R$ for all
	$K \in \mathcal{T}_h$, where $\rho_K$ is the radius of the largest
	ball ($d=2$) or sphere ($d=3$) inscribed in $K$. We will say that
	$(\mathfrak{T}_h,\mathfrak{F}_h)$ is \emph{quasi-uniform} if both the
	underlying spatial mesh $(\mathcal{T}_h,\mathcal{F}_h)$ is
	quasi-uniform, i.e. there exists a $C_U>0$ such that $h \le C_U h_K$
	for all $K \in \mathcal{T}_h$ and the time partition is quasi-uniform.
	We will say that $(\mathfrak{T}_h,\mathfrak{F}_h)$ is
	\emph{conforming} if the underlying spatial mesh
	$(\mathcal{T}_h,\mathcal{F}_h)$ is conforming: given two elements
	$K_1, K_2 \in \mathcal{T}_h$, either $K_1 \cap K_2 = \emptyset$ or
	$K_1 \cap K_2$ is a common vertex ($d=2$) or edge ($d=3$), or a common
	face of $K_1$ and $K_2$.  Finally, we will assume that every face $F$
	in the underlying spatial mesh $(\mathcal{T}_h, \mathcal{F}_h)$
	satisfies an \emph{equivalence condition}: that is, given
	$h_F = \text{diam}(F)$, there exist constants $C_e, C^e > 0$ such that
	$C_e h_K \le h_F \le C^e h_K$ for all $K \in \mathcal{T}_h$ and for
	all $F \in \mathcal{F}_h$ where $F \subset \partial K$.
	
	We now complete our definition of the space-time finite element spaces
	associated with the space-time prismatic mesh
	$(\mathfrak{T}_h,\mathfrak{F}_h)$ of $\Omega \times (0,T)$.  Let
	$k_t \ge 0$ be a fixed integer (not necessarily chosen to be equal to
	$k_s$). We consider the following tensor-product space-time finite
	element spaces in which we will seek our approximation on each
	space-time slab $\mathcal{E}^n$:
	\begin{equation*}
		\begin{split}
			\mathcal{V}_h &:= \cbr[1]{v_h \in L^2(0,T;L^2(\Omega)^d) \; | 
				\; v_h|_{\mathcal{E}^n}
				\in P_{k_t}(I_n;V_h)}, \\
			\mathcal{Q}_h &:= \cbr[1]{q_h \in L^2(0,T;L_0^2(\Omega)) \; | \; 
				q_h|_{\mathcal{E}^n}\in P_{k_t}(I_n;Q_h)}, \\
			\bar{\mathcal{V}}_h &:= \cbr[1]{\bar v_h \in 
				L^2(0,T;L^2(\partial \mathcal{T}_h)^d)
				\; | \; 
				\bar v_h|_{\mathcal{E}^n} \in P_{k_t}(I_n;\bar{V}_h)}, \\
			\bar{ \mathcal{Q}}_h &:= \cbr[1]{\bar q_h \in L^2(0,T;L^2(\partial 
				\mathcal{T}_h)) \; 
				| \; 
				\bar q_h|_{\mathcal{E}^n} \in P_{k_t}(I_n;\bar Q_h)}.
		\end{split}
	\end{equation*}
	
	As the spaces $\mathcal{V}_h$ and $\mathcal{Q}_h$ are non-conforming, we make use
	of \emph{broken differential operators}. For $v_h \in \mathcal{V}_h$, we introduce
	the \emph{broken gradient operator} $\nabla_h v_h$ by the restriction
	$(\nabla_h v_h)|_{K} = \nabla (v_h|_{K})$ and the \emph{broken time derivative}
	$\partial_\tau v_h$ by the restriction $(\partial_\tau v_h)|_{I_n} = \partial_t (v_h|_{I_n})$.
	Moreover, the trace of a function $v_h \in \mathcal{V}_h$ may be double-valued
	on space-time interior faces $\mathscr{F} \in \mathfrak{F}_h^i$ as well as across two space-time slabs $\mathcal{E}^n$ and $\mathcal{E}^{n+1}$. For fixed $n$, on an interior face $\mathscr{F} \in \mathfrak{F}_{h,n}^i$ shared by two space-time elements $\mathscr{K}^L$ and $\mathscr{K}^R$, we denote the traces of $v_h \in V_h$ on $\mathscr{F}$
	by $v_{h}^L = \text{ trace of } v_h|_{\mathscr{K}^L} \text{ on } \mathscr{F}$ and $v_{h}^R = \text{ trace of } v_h|_{\mathscr{K}^R} \text{ on } \mathscr{F}$. We denote by $u^{\pm}_n$ the traces at time level 
	$t_n$ from above and below, i.e. $u_n^{\pm} = \lim \limits_{\epsilon \searrow  0} 
	u_h(t_n \pm \epsilon)$.
	
	We introduce the \emph{jump} $\jump{\cdot}$ and  \emph{average} $\av{\cdot}$ of $v_h \in V_h$ across a space-time interior face $\mathscr{F}$ component-wise: let $\jump{v_{h,i}} = v_{h,i}^L - v_{h,i}^R$ and $\av{v_{h,i}} = (v_{h,i}^L + v_{h,i}^R)/2$ with $v_{h,i}$ denoting the $i$th Cartesian component of $v_h$. The quantities $\jump{v_h}$ and $\av{v_h}$ are then the vectors with $i$th Cartesian component $\jump{v_{h,i}}$ and $\av{v_{h,i}}$, respectively. On space-time boundary faces $\mathscr{F} \in \mathfrak{F}_h^b$, we set $\jump{v_h} = \av{v_h} = \text{ trace of } v_h|_{\mathscr{K}} \text{ on } \mathscr{F}$, where $\mathscr{K}$ is the element such that
	$\mathscr{F} \subset \partial \mathscr{K} \cap \sbr[0]{\partial \Omega \times (0,T)}$.
	Lastly, we define the \emph{time jump} of $v_h \in \mathcal{V}_h$ across the space-time slab $\mathcal{E}^n$ by $\sbr{v_h}_{n} = v_n^+ - v_n^-$.
	
	We adopt the following notation for various product spaces of interest
	in this work: $\boldsymbol{V}_h = V_h \times \bar{V}_h$,
	$\boldsymbol{Q}_h = Q_h \times \bar Q_h$,
	$\boldsymbol{\mathcal{V}}_h = \mathcal{V}_h \times
	\bar{\mathcal{V}}_h$, and
	$\boldsymbol{\mathcal{Q}}_h = \mathcal{Q}_h \times
	\bar{\mathcal{Q}}_h$. Pairs in these product spaces will be denoted
	using boldface; for example,
	$\boldsymbol{v}_h := (v_h, \bar{v}_h) \in \boldsymbol{\mathcal{V}}_h$.
	Lastly, we introduce two mesh-dependent norms on the spaces $V_h$ and
	$\boldsymbol{V}_h$, both of which are standard in the study of
	interior penalty methods:
	\begin{subequations}
		\begin{align*}
			\norm{v_h}_{1,h}^2 &:= \sum_{K\in\mathcal{T}_h} \norm{\nabla 
				v_h}_{K}^2 + \sum_{F \in \mathcal{F}_h} \frac{1}{h_F} \norm[0]{\jump{v_h}}_{L^2(F)}^2, &&\forall v_h \in V_h,
			\\
			\tnorm{\boldsymbol{v}}_v^2 &:= \sum_{K\in\mathcal{T}_h} \norm{\nabla 
				v_h}_{K}^2 
			+ \sum_{K\in\mathcal{T}_h}\frac{1}{h_K}\norm{v_h- \bar{v}_h}_{\partial 
				K}^2, &&\forall \boldsymbol{v}_h\in \boldsymbol{V}_h.
		\end{align*}
	\end{subequations}
	Throughout we use the notation $a \lesssim b$ to denote $a \le C b$
	where $C$ is a constant independent of the mesh parameters $h$ and
	$\tau$, the viscosity $\nu$, but possibly dependent on the polynomial
	degrees $k_t$ and $k_s$, the spatial dimension $d$, and the domain
	$\Omega$.
	
	Thanks to the equivalence condition on faces, we have
	\begin{equation} \label{eq:h1normbnd}
		\norm{v_h}_{1,h} \lesssim \tnorm{\boldsymbol{v}_h}_v, \quad \forall \boldsymbol{v}_h \in \boldsymbol{V}_h,
	\end{equation}
	and hence we
	can conclude the following discrete Poincar\'{e} inequality holds
	\cite[Corollary 5.4]{Pietro:book}: for all
	$\boldsymbol{v}_h \in \boldsymbol{V}_h$, 
	\begin{equation} \label{eq:disc_poinc}
	\norm{v_h}_{L^2(\Omega)} \lesssim \tnorm{\boldsymbol{v}_h}_v.
	\end{equation}

	\subsection{The space-time HDG method}
	
	In this subsection, we describe the fully discrete numerical method under consideration
	in this paper. We discretize the Navier--Stokes problem \cref{eq:ns_equations} using the 
	exactly mass conserving space-time HDG method on space-time prismatic meshes
	studied in \cite{Kirk:2021}. This method combines the point-wise
	divergence free and $H(\text{div};\Omega)$-conforming HDG method introduced in 
	\cite{Rhebergen:2018a} and analyzed in \cite{Kirk:2019b} with a discontinuous Galerkin 
	time stepping scheme; see also \cite{Horvath:2019,Horvath:2020} for related
	discretizations on space-time tetrahedral meshes on time-dependent domains.
	
	Due to the use of discontinuous-in-time finite element spaces, the
	discrete space-time HDG formulation can be localized to a single space-time slab; see e.g. \cite[Chapter 12]{Thomee:book}.
	We first consider the discrete formulation on a single space-time slab, and in \Cref{ss:equivalent_scheme} we will
	introduce the equivalent discrete formulation obtained by summing over all space-time slabs to aid us in our analysis.
	For $n=0,\dots, N-1$, the space-time HDG method for the 
	Navier--Stokes problem in each space-time slab $\mathcal{E}^n$ reads: find 
	$(\boldsymbol{u}_h, \boldsymbol{p}_h) \in \boldsymbol{\mathcal{V}}_h \times 
	\boldsymbol{\mathcal{Q}}_h$ such that for all test functions $(\boldsymbol{v}_h, 
	\boldsymbol{q}_h) \in \boldsymbol{\mathcal{V}}_h \times 
	\boldsymbol{\mathcal{Q}}_h$:
	\begin{multline}
		\label{eq:discrete_problem}
		-\int_{I_n} (u_h, \partial_t v_h)_{\mathcal{T}_h} \dif t + 
		(u_{n+1}^-,v_{n+1}^-)_{\mathcal{T}_h} +  
		\int_{I_n} \del{ \nu a_h(\boldsymbol{u}_h, \boldsymbol{v}_h) + o_h(u_h; 
			\boldsymbol{u}_h, \boldsymbol{v}_h) } \dif t
		\\ 
		+\int_{I_n} b_h(\boldsymbol{p}_h, v_h)  \dif 
		t  - \int_{I_n} b_h(\boldsymbol{q}_h, u_h) \dif t = (u_{n}^-,v_{n}^+)_{\mathcal{T}_h} + \int_{I_n} (f, 
		v_h)_{\mathcal{T}_h} 
		\dif t,
	\end{multline}
	where $(u, v)_{\mathcal{T}_h} = \sum_{K \in \mathcal{T}_h} (u, v)_{K}$.
	Once we have solved \cref{eq:discrete_problem} for $u_h$ in the space-time slab $\mathcal{E}^n$, the 
	trace $u_{n+1}^-$ serves as an initial condition when solving \cref{eq:discrete_problem} on 
	the next space-time slab $\mathcal{E}^{n+1}$. The process is initiated by choosing $u_0^{-} = \Pi_h^{\text{div}} u_0$ in the 
	first space-time slab $\mathcal{E}^0$, 
	where $u_0 \in H$ is the prescribed initial condition to the continuous problem \cref{eq:ns_equations}, and $\Pi_h^{\text{div}} : L^2(\Omega) \to V_h^{\text{div}}$ is the $L^2$-projection onto 
	the discretely divergence free subspace $V_h^{\text{div}} \subset V_h$; see \cref{eq:disc_div_free} below and the discussion following.
	
	The discrete
	multilinear forms $a_h(\cdot,\cdot): \boldsymbol{V}_h \times \boldsymbol{V}_h 
	\to \mathbb{R}$, $b_h(\cdot,\cdot): V_h \times \boldsymbol{Q}_h 
	\to \mathbb{R}$, and $o_h(\cdot;\cdot,\cdot): V_h \times \boldsymbol{V}_h 
	\times \boldsymbol{V}_h \to \mathbb{R}$ appearing in \cref{eq:discrete_problem} 
	serve as approximations to the viscous, pressure-velocity coupling, and 
	convection terms, respectively. We define them as in \cite{Kirk:2019b,Rhebergen:2018a}:

	\begin{subequations} \label{eq:forms}
		\begin{align}
			\label{eq:formA}
			a_h(\boldsymbol{u}, \boldsymbol{v}) :=&
			\sum_{K\in\mathcal{T}_h}\int_{K} \nabla u : \nabla v \dif x
			+ \sum_{K\in\mathcal{T}_h}\int_{\partial K}\frac{\alpha  }{h_K}(u - 
			\bar{u})\cdot(v - \bar{v}) \dif s
			\\
			\nonumber
			&- \sum_{K\in\mathcal{T}_h}\int_{\partial 
				K}\sbr{(u-\bar{u})\cdot\partial_n v
				+ \partial_nu\cdot(v-\bar{v})} \dif s,
			\\
			\label{eq:formO}
			o_h(w; \boldsymbol{u}, \boldsymbol{v}) :=&
			-\sum_{K\in\mathcal{T}_h}\int_{K} u\otimes w : \nabla v \dif x
			+\sum_{K\in\mathcal{T}_h}\int_{\partial K}\tfrac{1}{2}w\cdot 
			n\,(u+\bar{u})\cdot(v-\bar{v}) \dif s
			\\
			\nonumber
			&+\sum_{K\in\mathcal{T}_h}\int_{\partial K}\tfrac{1}{2}\envert{w\cdot 
				n}(u-\bar{u})\cdot(v-\bar{v})\dif s,
			\\
			\label{eq:formB}
			b_h(\boldsymbol{p}, v) :=&
			- \sum_{K\in\mathcal{T}_h}\int_{K} p \nabla \cdot v \dif x
			+ \sum_{K\in\mathcal{T}_h}\int_{\partial K}v \cdot n \, \bar{p} \dif s.
		\end{align}
	\end{subequations}
	The parameter $\alpha > 0$ appearing in the bilinear form $a_h(\cdot,\cdot)$ is a penalty parameter typical of interior penalty type 
	discretizations. The bilinear form
	$a_h(\cdot, \cdot)$ is continuous and for sufficiently large $\alpha$
	enjoys discrete coercivity \cite[Lemmas 4.2 and 4.3]{Rhebergen:2017},
	i.e. for all $\boldsymbol{u}_h, \boldsymbol{v}_h \in\boldsymbol{V}_h$,
	\begin{equation}
		\label{eq:ah_coer_bnd}
		\tnorm{\boldsymbol{v}_h}_v^2 \lesssim a_h(\boldsymbol{v}_h, \boldsymbol{v}_h) \quad 
		\text{and} \quad
		\envert{a_h(\boldsymbol{u}_h, \boldsymbol{v}_h)} \lesssim \tnorm{\boldsymbol{u}_h}_{v}\tnorm{\boldsymbol{v}_h}_{v}.
	\end{equation}
	The trilinear form $o_h(\cdot;\cdot,\cdot)$ satisfies \cite[Proposition 3.6]{Cesmelioglu:2017}
	\begin{equation}
		\label{eq:ohequals}
		o_h(w_h; \boldsymbol{v}_h, \boldsymbol{v}_h) = 
		\frac{1}{2}\sum_{K\in\mathcal{T}}\int_{\partial K}\envert{w_h\cdot 
			n}\envert{v_h - \bar{v}_h}^2 \dif s \ge 0, \qquad w_h \in V_h^{\text{div}}, \;
		\forall \boldsymbol{v}_h \in\boldsymbol{V}_h.
	\end{equation}
	%
	%
	The trilinear form also satisfies \cite{Kirk:2021} for all
	$\boldsymbol{u}_h,\boldsymbol{v}_h \in \boldsymbol{V}_h$ and
	$d \in \cbr{2,3}$,
	\begin{equation} \label{eq:o_space_bnd}
		\begin{split}
			|o_h(u_h;\boldsymbol{u}_h,\boldsymbol{v}_h)|  
			\lesssim \norm{u_h}_{L^2(\Omega)}^{1/(d-1)} \tnorm{\boldsymbol{u}_h}_v^{d/2} \tnorm{\boldsymbol{v}_h}_v, \quad \forall \boldsymbol{u}_h,\boldsymbol{v}_h \in \boldsymbol{\mathcal{V}}_h.
		\end{split}
	\end{equation}
	%
	%
	
	Frequent 
	use will
	also be made of functions in the subspace of discretely divergence free 
	velocity 
	fields:
	\begin{equation}  \label{eq:disc_div_free}
		\begin{split}
			V_h^{\text{div}} &:= \cbr{v_h \in V_h \; : \; b_h(v_h,\boldsymbol{q}_h) = 0, 
				\;
				\forall \boldsymbol{q}_h \in \boldsymbol{Q}_h}, \\
			\mathcal{V}_h^{\text{div}} &:= \cbr[0]{v_h \in \mathcal{V}_h \; : \; \int_0^T b_h(v_h,\boldsymbol{q}_h) \dif t = 0, 
				\;
				\forall \boldsymbol{q}_h \in \boldsymbol{\mathcal{Q}}_h}.
		\end{split}
	\end{equation}
	We note that $V_h^{\text{div}} \subset H(\text{div};\Omega)$, and further $\nabla \cdot v_h = 0$ and $v_h \cdot n|_{\partial \Omega} = 0$
	for all $v_h \in V_h^{\text{div}}$ (see e.g. \cite[Proposition 1]{Rhebergen:2018a}). As $\Omega$ is assumed to have a Lipschitz boundary,
	we therefore have  $V_h^{\text{div}} \subset H$. In fact, it can be shown that $V_h^{\text{div}} = V_h \cap H$. 
	
	\subsection{Properties of the space-time HDG scheme}
	
	Here, we collect a number of useful results concerning the solution of
	the space-time HDG scheme \cref{eq:discrete_problem}. The existence of
	solutions to the nonlinear algebraic system arising from
	\cref{eq:discrete_problem} was shown in \cite{Kirk:2021}. It was also
	proven in \cite{Kirk:2021} that the discrete velocity $u_h$ computed
	using the space-time HDG scheme is \emph{conforming} in $L^2(0,T;H)$,
	i.e., if $u_h \in \mathcal{V}_h$ is the element velocity solution
	of \cref{eq:discrete_problem}, then $\nabla \cdot u_h = 0$,
	$u_h|_{\mathcal{E}^n} \in P_k(I_n;H(\text{div};\Omega))$, and the
	normal trace of $u_h$ vanishes on the spatial boundary $\partial \Omega$.
	
	Next, we recall an energy estimate that will allow us to conclude
	that the discrete velocity pair
	$\boldsymbol{u}_h \in \boldsymbol{\mathcal{V}}_h$ computed using
	\cref{eq:discrete_problem} is bounded uniformly with respect to the
	mesh parameters $h$ and $\tau$:
	\begin{lemma} \label{lem:energy_stability}
		Let $d \in \cbr{2,3}$, $k_s \ge 1$ and $k_t \ge 0$, and suppose that $\boldsymbol{u}_h \in 
		\boldsymbol{\mathcal{V}}_h$ is solution of space-time HDG scheme 
		\cref{eq:discrete_problem} for $n=0,\dots,N-1$. For all $0 \le m \le N-1$,
		\begin{equation} \label{eq:energy_1}
			\norm[1]{u_{m+1}^-}_{L^2(\Omega)}^2 
			+ \sum_{n=0}^{m} \norm[1]{ 
				\sbr{u_h}_n}_{L^2(\Omega)}^2 +
			\nu \int_{0}^{t_{m+1}}\tnorm{\boldsymbol{u}_h}_v^2 \dif t  \le C(f, u_0, \nu).
		\end{equation}
		Furthermore, if $k_t \ge 0$ when $d=2$ and $k_t \in \cbr{0,1}$ when $d=3$,
		it holds that
		\begin{equation} \label{eq:energy_2}
			\norm{u_h}_{L^{\infty}(0,T;L^2(\Omega)^d)} \le C(f, u_0, \nu).
		\end{equation}
		Here, $C(f, u_0, \nu)$ denotes a constant that depends on the data $f$,
		$u_0$, and $\nu$.
	\end{lemma}
	The bounds in \Cref{lem:energy_stability} were proven in \cite{Kirk:2021} under the
	assumption that $k_t = k_s \ge 1$ for simplicity of presentation; we
	remark that the proofs are equally valid for the general case
	$k_s \ge 1$ and $k_t \ge 0$. Note that for the lower order schemes $k_t \in \cbr{0,1}$,
	\cref{eq:energy_2} follows directly from the \cref{eq:energy_1}. This can be seen immediately when considering
	constant polynomials in time ($k_t = 0$).  For linear polynomials in
	time ($k_t = 1$), this follows from the bound (see \cite[Section
	3]{Walkington:2005}):
	$\norm{u_h}_{L^{\infty}(0,T;L^2(\Omega)^d)} \le \max_{0 \le m \le N-1}
	\norm[0]{u_{m+1}^-}_{L^2(\Omega)} + \max_{0 \le m \le N-1}
	\norm[0]{\sbr{u_h}_m}_{L^2(\Omega)}$.

	\section{Lifting operators and discrete differential operators} \label{s:disc_operators}
	In this section, we introduce two discrete differential operators that serve as natural 
	approximations to the distributional
	gradient and distributional time derivative in the space-time HDG setting.
	These discrete operators enjoy convergence to their continuous counterparts 
	in the weak topologies of appropriate Bochner spaces. 
	
	\subsection{Discrete gradient}
	\label{ss:dgo}
	First, we introduce a discrete gradient operator that will serve as an
	approximation of the distributional gradient operator following ideas in \cite{Buffa:2009,DiPietro:2010, Kikuchi:2012}. 
	%
	The basic building 
	block of the discrete gradient operator is the following observation \cite{Buffa:2009}: as functions $v_h \in V_h$ are discontinuous, their distributional gradient
	has a contribution from the jumps of $u_h$ across element interfaces. 
	Therefore, an appropriate approximation of the distributional gradient
	in the HDG setting must incorporate the contribution from the jumps
	between the element solution and the facet solution across element
	boundaries.  We do so by constructing an HDG lifting operator
	following ideas in \cite{Kikuchi:2012,Oikawa:2010}.  For this, we need
	to introduce the scalar broken polynomial spaces
	$W_h := \cbr[1]{w_h \in L^2(\Omega) \; | \; w_h|_{K} \in P_{k_s}(K)}$
	and
	$\bar{W}_h := \cbr[1]{w_h \in L^2(\partial \mathcal{T}_h) \; | \;
		w_h|_{\partial K} \in P_{k_s}(\partial K)}$.
	We first define a \emph{local lifting} $R_h^{\partial K} : L^2(\partial K) \to P_{k_s}(K)^d$ 
	satisfying 
	\begin{equation}
		\int_{K} R_h^{\partial K}(\mu) \cdot w_h \dif x = \int_{\partial K} \mu w_h 
		\cdot n \dif s, \quad \forall w_h \in P_{k_s}(K)^d.
	\end{equation}
	We then define the \emph{global lifting} $R^{k_s}_h: L^2(\partial \mathcal{T}_h) \to V_h$  
	by the restriction
	$R_h^{k_s}(\mu)|_{K} = R_h^{\partial K}(\mu|_{\partial K})$ for all $K \in 
	\mathcal{T}_h$. Note that $R^{k_s}_h$ satisfies for all  
	$w_h 
	\in V_h$,
	\begin{equation} \label{eq:global_spat_lift}
		\sum_{K \in \mathcal{T}_h} \int_K R_h^{k_s}(\mu) \cdot w_h \dif x = \sum_{K \in 
			\mathcal{T}_h} \int_{\partial K} \mu w_h \cdot n \dif s, 
	\end{equation}
	and it can be shown using the Cauchy--Schwarz inequality
	and a standard local discrete trace inequality that 
	\begin{equation} \label{eq:bnd_on_lift}
		\norm[0]{R_h^{k_s}(w_{h} - 
			\bar{w}_h)}_{L^2(\Omega)}^2 
		\lesssim \sum_{K \in \mathcal{T}_h} \frac{1}{h_K} \norm{w_h - 
			\bar{w}_{h}}_{L^2(\partial K)}^2, \quad
		\forall \boldsymbol{w}_h \in W_h \times \bar{W}_h.
	\end{equation}
	%
	%
	%
	
	Using the global HDG lifting, we introduce the
	discrete gradient operator $G_h^{k_s}: 
	W_h \times \bar{W}_h
	\to V_h $ in the same spirit as in 
	\cite{DiPietro:2010, Kikuchi:2012}: given $(v,\bar{v}) \in W_h \times \bar{W}_h$,
	we set
	\begin{equation} \label{eq:disc_grad_op}
		G_h^{k_s}(v,\bar{v}) = \nabla_h v - R_h^{k_s}(v - \bar{v}),
	\end{equation}
	where $\nabla_h$ is the broken  gradient operator. Crucially, this operator satisfies
	for all $\boldsymbol{v}_h \in \boldsymbol{V}_h$ and $w_h \in V_h$ the identity
	\begin{equation*}
		\int_{\Omega} G_h^{k_s}(\boldsymbol{v}_{h,i}) \cdot w_h \dif x = \int_{\Omega} \nabla_h v_{h,i} \cdot w_h \dif x - \sum_{K \in \mathcal{T}_h} \int_{\partial K}(v _{h,i}- \bar{v}_{h,i}) w_h\cdot n \dif s,
	\end{equation*}
	where $v_{h,i}$ and $\bar{v}_{h,i}$ denote the $i$th Cartesian components
	of $v_h$ and $\bar{v}_h$, respectively.

	\subsection{Discrete time derivative}
	\label{ss:dtd}
	To define a discrete time derivative operator that serves as an appropriate
	approximation of the distributional time derivative, we proceed by analogy with the discrete gradient
	constructed in the previous section. 
	We follow \cite{Schotzau:2010} by introducing a \emph{local time lifting operator} $R_{loc,n}^{k_t}: V_h \to P_{k_t}(I_n;V_h)$
	satisfying
	\begin{align}
		\int_{I_n} (R_{loc,n}^{k_t}(u_h),v_h)_{\mathcal{T}_h} \dif t
		&= 
		([u_h]_n,v_n^+)_{\mathcal{T}_h}, \quad \forall v_h \in \; P_{k_t}(I_n;V_h),
		\\
		\label{eq:time_lift_representation}
		R^{k_t}_{loc,n}(u_h)
		&= \frac{\del{u_n^+ - u_n^-}}{2} \sum_{m=0}^{k_t} (-1)^m(2m+1) L_m^n(t),
	\end{align}
	where the latter representation formula follows from \cite[Lemma
	6]{Schotzau:2010}. Here $L_m^n(t)$ are \emph{mapped Legendre
		polynomials}; see \cite[Section 3]{Schotzau:2010}.  We then define a
	\emph{global time lifting} $R^{k_t}: V_h \to \mathcal{V}_h$ by the
	restriction $R^{k_t}|_{I_n} = R^{k_t}_{loc,n}$. This lifting
	satisfies:
	\begin{equation}
		\begin{split} \label{eq:global_time_lift}
			\int_0^T (R^{k_t}(u_h),v_h)_{\mathcal{T}_h} \dif t &= \sum_{n=0}^{N-1} 
			([u_h]_n,v_n^+)_{\mathcal{T}_h}, \quad \forall v_h \in \mathcal{V}_h.
		\end{split}
	\end{equation}
	With the global time lifting in hand, we define the discrete time
	derivative $\mathcal{D}^{k_t}_t: \mathcal{V}_h \to \mathcal{V}_h$ of
	$v_h \in \mathcal{V}_h$ by setting
	\begin{equation} \label{eq:discrete_time_der}
		\mathcal{D}^{k_t}_t (v_h) = \partial_\tau v_h + R^{k_t}(v_h).
	\end{equation}
	
	\begin{lemma} \label{lem:disc_time_der_in_H}
		Suppose that $u_h \in \mathcal{V}_h^{\text{div}}$. Then, it 
		holds 
		that $\mathcal{D}^{k_t}_t (u_h)|_{\mathcal{E}^n}\in P_{k_t}(I_n;H)$ for all $0 \le n \le N-1$.
	\end{lemma}
	\begin{proof}
		That $\mathcal{D}^{k_t}_t(u_h)$ is divergence free and
		$H(\text{div};\Omega)$-conforming follows from the fact that the
		broken time derivative commutes with the divergence operator and the
		representation formula \cref{eq:time_lift_representation}. It
		remains to show that
		$\mathcal{D}_t^{k_t}(u_h) \cdot n|_{\partial \Omega}$ = 0.  Note
		that $u_h \cdot n |_{\partial \Omega} = 0$ implies
		$(\partial_\tau u_h )\cdot n|_{\partial \Omega} = 0$. This can be
		seen by considering a single space-time slab $\mathcal{E}^n$ and
		expanding $u_h$ in terms of a basis $\cbr{ \psi_i}_{i=0}^{m}$ of
		$P_{k_t}(I_n)$ to find
		$(\partial_\tau u_h)|_{I_n} = \sum_{i=0}^{k_t} \partial_t \psi_i
		u_i$, where $u_i \in V_h$ is such that
		$u_i \cdot n |_{\partial \Omega} = 0$. Lastly, since
		$u_n^+ \cdot n|_{\partial \Omega} = u_n^- \cdot n|_{\partial \Omega}
		= 0$, the representation formula \cref{eq:time_lift_representation}
		shows that indeed
		$R^{k_t}_{loc,n}(u_h) \cdot n|_{\partial \Omega} = 0$.
	\end{proof}
	
	\subsection{Rewriting the HDG scheme}
	\label{ss:equivalent_scheme}
	
	We now recast the space-time HDG scheme into a form more amenable to
	the convergence analysis in \Cref{s:convergence} using the discrete
	differential operators introduced above.  In what follows,
	$(u_{h,i})_{1\le i \le d}$, $(\bar u_{h,i})_{1\le i \le d}$,
	$(v_{h,i})_{1\le i \le d}$ and $(\bar{v}_{h,i})_{1\le i \le d}$ will
	denote the Cartesian components of $u_h$, $\bar{u}_h$, $v_h$ and
	$\bar{v}_h$, respectively. We will adopt the convention of summation
	over repeated indices. Restricting our attention to test functions
	$\boldsymbol{v}_h \in \mathcal{V}_h^{\text{div}} \times
	\bar{\mathcal{V}}_h$ in \cref{eq:discrete_problem} to remove the
	contribution from the pressure-velocity coupling term, integrating by
	parts in time, summing over all space-time slabs $\mathcal{E}^n$, and
	using the definitions of the lifting operators, we arrive at the
	problem:
	find
	$\boldsymbol{u}_h \in \mathcal{V}_h^{\text{div}} \times \bar{\mathcal{V}}_h$ satisfying for all
	$\boldsymbol{v}_h \in \mathcal{V}_h^{\text{div}} \times \bar{\mathcal{V}}_h$,
	\begin{equation} \label{eq:scheme_rewritten}
		\begin{split}
			\int_{0}^T &( \mathcal{D}_t^{k_t} (u_h),v_h)_{\mathcal{T}_h} \dif t  +
			\int_{0}^T \del{ \nu a_h(\boldsymbol{u}_h, \boldsymbol{v}_h) + o_h(u_h; 
				\boldsymbol{u}_h, \boldsymbol{v}_h) } \dif t =   \int_{0}^T (f, 
			v_h)_{\mathcal{T}_h} 
			\dif t,
		\end{split}
	\end{equation}
	where
	%
	\begin{align} 
		a_h(\boldsymbol{u}_h, \boldsymbol{v}_h)  \label{eq:viscous_rewritten}
		=&  \int_{\Omega} G_h^{k_s}(\boldsymbol{u}_{h,i}) 
		\cdot 
		G_h^{k_s}(\boldsymbol{v}_{h,i}) \dif x 
		- \int_{\Omega}  R_h^{k_s}(u_{h,i}-\bar{u}_{h,i}) 
		\cdot  R_h^{k_s}(v_{h,i}-\bar{v}_{h,i})  \dif x 
		\\ &+ \notag
		\sum_{K\in\mathcal{T}_h} \int_{\partial K}\frac{\alpha  
		}{h_K}(u_{h,i} - \bar{u}_{h,i})(v_{h,i}-\bar{v}_{h,i})\dif s, \\
		o_h(u_h; \boldsymbol{u}_h, \boldsymbol{v}_h) \dif t =& 
		\int_{\Omega} u_h \cdot 
		G_h^{2k_s}(\boldsymbol{u}_{h,i}) 
		v_{h,i} 	\label{eq:convection_rewritten}
		\dif x \\ \notag
		&+\sum_{K\in\mathcal{T}_h} \int_{\partial 
			K}\tfrac{1}{2}\del{u_h 
			\cdot n + 
			\envert{u_h\cdot 
				n}}(u_h-\bar{u}_h)\cdot(v_{h}-\bar{v}_{h})\dif s.
	\end{align}
	\section{Uniform bounds on the discrete differential
		operators} \label{s:bnds_compact} In this section, we derive uniform
	bounds on the discrete differential operators of the discrete velocity
	solution introduced in the previous section.  In what follows, we
	suppose that
	$\boldsymbol{u}_h \in \mathcal{V}_h^{\text{div}} \times
	\bar{\mathcal{V}}_h$ is a discrete velocity pair solving the
	space-time HDG formulation \cref{eq:discrete_problem} on a given
	space-time prismatic mesh $\del{\mathfrak{T}_h, \mathfrak{F}_h}$ for
	$n=0,\dots,N-1$. 
	We then show that subsequences of the discrete
	derivatives converge weakly to their continuous counterparts.
	
	\subsection{Bounding the discrete gradient}
	\label{ss:grad_bnd}
	Before bounding the discrete gradient of $u_h$, we pause to mention an
	immediate consequence of the energy bound \Cref{lem:energy_stability}.
	From the discrete Sobolev embeddings for broken polynomial spaces
	\cite[Theorem 6.1]{DiPietro:2010}, we can infer using
	\cref{eq:h1normbnd} that
	\begin{equation} \label{eq:lq_bnd_uh}
		\int_0^T \norm{u_h}_{L^q(\Omega)}^2 \dif t  \le C(f, u_0, \nu),
	\end{equation}
	where $1 \le q < \infty$ if $d=2$ and $1 \le q \le 6$ if $d=3$.
	Thus, $\del{u_h}_{h \in \mathcal{H}}$ is bounded in $L^2(0,T;L^q(\Omega)^d)$ 
	for $1 \le q \le 6$ and in particular in $L^2(0,T;H)$. 
	
	\begin{theorem}\label{thm:bnd_disc_grad}
		Let $\del{\mathfrak{T}_h, \mathfrak{F}_h}$ be a conforming and shape-regular space-time prismatic mesh.
		Let $d \in \cbr{2,3}$ and suppose $k_t \ge 0$ if $d=2$ and $k_t \in \cbr{0,1}$ if $d=3$. Let $\boldsymbol{u}_h$ be the
		solution of the space-time HDG scheme \cref{eq:discrete_problem}.
		Then, provided the penalty parameter $\alpha > 0$ is chosen sufficiently large, it holds that
		\begin{equation}
			\int_{0}^T 
			\norm[0]{G_h^k(\boldsymbol{u}_{h,i})}_{L^2(\Omega)}^2 
			\dif t \le C(f,u_0,\nu).
		\end{equation}
	\end{theorem}
	
	\begin{proof}
		The result follows from \cref{eq:viscous_rewritten} and the energy bound in \Cref{lem:energy_stability}, provided $\alpha >0$ is chosen sufficiently large, since
		for all $\boldsymbol{u}_h \in \boldsymbol{V}_h$ we have by \cref{eq:bnd_on_lift} for $i=1,\dots,d$ that
		\begin{multline*}
			-  \norm[1]{R_h^{k_s}(u_{h,i}-\bar{u}_{h,i})}_{L^2(\Omega)}^2 + 
			\sum_{K\in\mathcal{T}_h} \frac{\alpha  
			}{h_K} \norm{u_{h,i} - \bar{u}_{h,i}}_{L^2(\partial K)}^2
			\\\ge \del{\alpha - C} \sum_{K\in\mathcal{T}_h} \frac{1 
			}{h_K} \norm{u_{h,i} - \bar{u}_{h,i}}_{L^2(\partial K)}^2,  
		\end{multline*}
		and therefore, 
		\begin{equation} \label{eq:lower_bnd_a_h}
			\begin{split}
				a_h(\boldsymbol{u}_h, \boldsymbol{u}_h) \ge \sum_{i=1}^d \norm[0]{G_h^{k_s}(\boldsymbol{u}_{h,i})}_{L^2(\Omega)}^2.
			\end{split}
		\end{equation}
		Consequently, the sequence $G_h^k(\boldsymbol{u}_{h,i})$ is 
		bounded in $L^2(0,T;L^2(\Omega)^d)$. 
	\end{proof}
	
	\subsection{Bounding the discrete time derivative}
	We now turn our focus to bounding the discrete time derivative of
	$u_h \in \mathcal{V}_h^{\text{div}}$ uniformly, first in the dual
	space of $\mathcal{V}_h^{\text{div}} \times \bar{\mathcal{V}}_h$, and
	second in $L^{4/d}(0,T;V')$. The former is required to obtain a
	\emph{strong compactness} result needed for passage to the limit as
	$h\to 0$ in the nonlinear convection term, and the second is essential
	to ensure the distributional time derivatives of accumulation points
	of the sequence $\cbr{u_h}_{h \in \mathcal{H}}$ are sufficiently
	regular to satisfy \Cref{def:equiv_weak_form}.  That
	$\mathcal{D}_t^{k_t}(u_h)$ can be identified with an element of
	$L^{4/d}(0,T;V')$ follows from \Cref{lem:disc_time_der_in_H} since
	$V \subset H$ with continuous and dense embedding
	(see e.g. \cite[Chapter 5.2]{Brezis:book}).
	\subsubsection{Uniform bound in the dual space of $\mathcal{V}_h^{\text{div}} \times \bar{\mathcal{V}}_h$}
	To apply the compactness theorem \Cref{thm:compactness} later on to prove \Cref{thm:convergence_of_subsequence},
	we will require $F_h: \boldsymbol{v}_h \mapsto (\mathcal{D}_t^{k_t}(u_h),v_h)_{\mathcal{T}_h}$ to be uniformly bounded $L^{4/d}(0,T;(V_h^{\text{div}} \times \bar{V}_h)')$,
	with $(V_h^{\text{div}} \times \bar{V}_h)'$ the dual space of $V_h^{\text{div}} \times \bar{V}_h$. We shall see that it suffices to
	bound $F_h(\boldsymbol{v}_h)$ in the dual space of the fully discrete space $\mathcal{V}_{h}^{\text{div}} \times \bar{\mathcal{V}}_h$,
	which we equip with the norm
	\begin{equation*}
		\norm{F_h}_{(\mathcal{V}_h^{\text{div}} \times \bar{\mathcal{V}}_h)'} =
		\sup_{0 \ne \boldsymbol{v}_h \in \mathcal{V}_h^{\text{div}} \times \bar{\mathcal{V}}_h}   \frac{\big|\int_0^T F_h(\boldsymbol{v}_h) \dif t\big|}{ \del{ \int_0^T \tnorm{\boldsymbol{v}_h}^{4/(4-d)}_v \dif t}^{(4-d)/4}}.
	\end{equation*}
	This motivates the following result (where we choose $F_h: \boldsymbol{v}_h \mapsto ( \mathcal{D}_t^{k_t}(u_h), v_h)_{\mathcal{T}_h}$):
	\begin{lemma} \label{lem:disc_time_der_bnd_disc_dual}
		Let $\del{\mathfrak{T}_h, \mathfrak{F}_h}$ be a conforming and shape-regular space-time prismatic mesh.
		Let $d \in \cbr{2,3}$ and suppose $k_t \ge 0$ if $d=2$ and $k_t \in \cbr{0,1}$ if $d=3$. Let $\boldsymbol{u}_h$ be the discrete velocity pair arising from the
		solution of the space-time HDG scheme \cref{eq:discrete_problem}.
		It holds for all $\boldsymbol{v}_h \in \mathcal{V}_h^{\text{div}} \times \bar{\mathcal{V}}_h$ that
		\begin{equation*}
			\begin{split}
				\bigg| \int_0^T ( \mathcal{D}_t^{k_t}(u_h), v_h)_{\mathcal{T}_h} \dif t \bigg| \le C(u_0,f,\nu,T)\del[3]{ \int_0^T \tnorm{\boldsymbol{v}_h}^{4/(4-d)}_v \dif t}^{(4-d)/4}.
			\end{split}
		\end{equation*}
	\end{lemma}
	\begin{proof}
		Let
		$\boldsymbol{v}_h \in \mathcal{V}_h^{\text{div}} \times
		\bar{\mathcal{V}}_h$, and use \cref{eq:scheme_rewritten} to write
		\begin{equation} \label{eq:bnd_time_der_inter}
			\begin{split}
				\int_{0}^T (\mathcal{D}_t^{k_t} u_h 
				, v_h)_{\mathcal{T}_h} 
				\dif t = \int_{0}^{T} \del{(f, 
				v_h)_{\mathcal{T}_h}- \nu 
				a_h(\boldsymbol{u}_h, 
				\boldsymbol{v}_h) - o_h(u_h; 
				\boldsymbol{u}_h, \boldsymbol{v}_h)}
				\dif t.
			\end{split}
		\end{equation}
		We now bound each of the three terms on the right hand side of 
		\cref{eq:bnd_time_der_inter}, beginning with the first term
		on the right hand side. The Cauchy-Schwarz 
		inequality, H\"{o}lder's inequality, and the discrete Poincar\'{e} inequality \cref{eq:disc_poinc} yield
		\begin{equation} \label{eq:f_bnd}
			\begin{split}
				\int_{0}^{T} |(f, 
				v_h)_{\mathcal{T}_h}|\dif t 
				& \le C(f,T) \del[3]{ \int_0^T\tnorm{\boldsymbol{v}_h}^{4/(4-d)}_v \dif t}^{(4-d)/4}.
			\end{split}
		\end{equation}
		To bound the linear viscous term on the right hand side of \cref{eq:bnd_time_der_inter},
		we begin by using the boundedness of $a_h(\cdot,\cdot)$ \cref{eq:ah_coer_bnd}
		and H\"{o}lder's inequality with $p = 4/d$ and $q = 4/(4-d)$ to find
		\begin{equation} \label{eq:ah_bnd_inter}
			\begin{split}
				\int_{0}^T |a_h(\boldsymbol{u}_h, 
				\boldsymbol{v}_h) |\dif t &\le C \del[3]{ \int_{0}^T 
					\tnorm{\boldsymbol{u}_h}_v^{4/d} \dif t
				}^{d/4} \del[3]{ \int_{0}^T 
					\tnorm{\boldsymbol{v}_h}_v^{4/(4-d)} \dif t
				}^{(4-d)/4}.
			\end{split}
		\end{equation}
		If $d=2$, directly using the uniform bound in \Cref{lem:energy_stability}, and if $d=3$, applying H\"{o}lder's inequality to the first integral on the right hand side of \cref {eq:ah_bnd_inter} with $p=3$ and $q=3/2$,
		followed by the uniform bound in \Cref{lem:energy_stability}, we find
		\begin{equation} \label{eq:ah_bnd_3d}
			\begin{split}
				\int_{0}^T |a_h(\boldsymbol{u}_h, 
				\boldsymbol{v}_h) |\dif t 
				&\le C(f,u_0,\nu,T)\del[3]{ \int_{0}^T 
					\tnorm{\boldsymbol{v}_h}_v^{4/(4-d)} \dif t
				}^{(4-d)/4}.
			\end{split}
		\end{equation}
		Lastly, we must bound the nonlinear convection term on
		the right hand side of \cref{eq:bnd_time_der_inter}.
		For this, we use the bound \cref{eq:o_space_bnd},
		apply the generalized H\"{o}lder's inequality with $p = \infty$, $q = 4/d$, and $r= 4/(4-d)$,
		and use \Cref{lem:energy_stability}, to find
		\begin{equation} \label{eq:oh_bnd}
			\int_{0}^T |o_h(u_h;\boldsymbol{u}_h,\boldsymbol{v}_h)| \dif t \le C(f,u_0,\nu)\del{ \int_{0}^T 
				\tnorm{\boldsymbol{v}_h}_v^{4/(4-d)} \dif t  }^{(4-d)/4}.
		\end{equation}
		Collecting \cref{eq:f_bnd},  \cref{eq:ah_bnd_3d}, and \cref{eq:oh_bnd} yields the result.
	\end{proof}
	
	\subsubsection{Construction of suitable test functions}
	
	To prove a uniform bound on the discrete time derivative in
	$L^{4/d}(0,T; V')$ (see \Cref{thm:disc_time_der_bnd}), we will need to
	construct a suitable set of test functions in the discrete space
	$\mathcal{V}_h^{\text{div}} \times \bar{\mathcal{V}}_h$.  This will
	require two preparatory results. The first is a density result for
	functions of tensor-product type in $C_c(0,T;V)$ taken from
	\cite[Lemma V.1.2]{Boyer:book} with minor modification:
	\begin{lemma} \label{lem:tensor_product_functions}
		The set $\mathscr{M}$ of functions $\varphi$ of the form 
		\begin{equation} \label{eq:tensor_product_functions}
			\varphi(t,x) = \sum_{k=1}^{M} \eta_k(t) \psi_k(x),
		\end{equation}
		where $M \ge 1$ is an integer, $\eta_k \in C^{\infty}_c(0,T)$, and $\psi_k \in \mathscr{V}$, is dense in $C_c(0,T;V)$.
	\end{lemma}
	
	Denote by
	$\Pi_n^t: L^2(I_n) \to P_{k_t}(I_n)$,
	$\Pi_h^{\text{div}}: L^2(\Omega) \to V_h^{\text{div}}$, and
	$\bar{\Pi}_h: H^1(\Omega)^d \to \bar{V}_h$ the orthogonal
	$L^2$-projections onto the discrete spaces $P_{k_t}(I_n)$,
	$V_h^{\text{div}}$, and $\bar{V}_h$, respectively. We define the
	\emph{global} $L^2$-projection $\Pi^t$ in time by the restriction
	$\Pi^t|_{I_n} = \Pi^t_n$.  Given a function $\varphi \in \mathscr{M}$,
	consider for all
	$n=0,\dots,N-1$,
	\begin{equation} \label{eq:test_func_for_time_der_bnd}
		\Pi \varphi|_{\mathcal{E}^n} = \sum_{k=1}^M \Pi_n^t \eta_k(t) \Pi_h^{\text{div}} \psi_k(x)
		\quad \text{ and } 
		\quad
		\bar{\Pi} \varphi |_{\mathcal{E}^n} = \sum_{k=1}^M \Pi_n^t \eta_k(t) \bar{\Pi}_h \psi_k(x).
	\end{equation}
	By construction, $(\Pi v, \bar{\Pi} v) \in \mathcal{V}_h^{\text{div}} \times \bar{\mathcal{V}}_h$. 
	We remark that the approximation properties of $\Pi_h^{\text{div}}$
	obtained in \cite{Kirk:2021} and listed in \Cref{lem:approximation_div} require quasi-uniformity of the underlying
	spatial mesh $\mathcal{T}_h$, which we assume henceforth wherever
	necessary.

	\begin{proposition} \label{prop:proj_bnd}
		Let $\del{\mathfrak{T}_h, \mathfrak{F}_h}$ be a conforming and quasi-uniform space-time prismatic mesh,
		and let $d \in \cbr{2,3}$.
		Let $\varphi \in \mathscr{M}$ and let $(\Pi \varphi, \bar{\Pi} \varphi) \in \mathcal{V}_h^{\text{div}} \times \bar{\mathcal{V}}_h$ 
		be the discrete test functions constructed in \cref{eq:test_func_for_time_der_bnd}.
		Then, the following stability property holds:
		\begin{align}
			\label{eq:proj_bnd}
			\int_{0}^T \tnorm{(\Pi \varphi, \bar{\Pi} \varphi)}_v^{4/(4-d)} \dif t &\lesssim \int_{0}^T \norm{\varphi}_{V}^{4/(4-d)}\dif t, \quad \forall v \in \mathscr{M}.
		\end{align}
	\end{proposition}
	\begin{proof}
		See \Cref{ss:proj_bnd_proof}.
	\end{proof}
	\subsubsection{Uniform bound in $L^{4/d}(0,T;V')$}
	
	With
	\Cref{lem:disc_time_der_bnd_disc_dual,lem:tensor_product_functions},
	and \Cref{prop:proj_bnd} in hand, we can now prove the main result of
	this subsection. Since $V'$ is separable, we can identify
	$L^{4/d}(0,T;V') \cong L^{4/(4-d)}(0,T;V)'$ (see
	e.g. \cite[Proposition 1.38]{Roubicek:book}), and since $(V,H,V')$
	form a Gelfand triple, we have
	\begin{equation} \label{eq:norm_duality_prod}
		\norm[0]{ \mathcal{D}_t^{k_t}(u_h)}_{L^{4/d}(0,T;V')}  = \sup_{ 0 \ne v \in L^{4/(4-d)}(0,T;V)} \frac{\big| \int_0^T ( \mathcal{D}_t^{k_t}(u_h),v)_{\mathcal{T}_h} \dif t \big|}{\norm{v}_{L^{4/(4-d)}(0,T;V)}}.
	\end{equation}
	
	\begin{theorem}[Uniform bound on the discrete time derivative] \label{thm:disc_time_der_bnd}
		Let $\del{\mathfrak{T}_h, \mathfrak{F}_h}$ be a conforming and
		quasi-uniform space-time prismatic mesh.  Let $d \in \cbr{2,3}$ and
		suppose $k_t \ge 0$ if $d=2$ and $k_t \in \cbr{0,1}$ if $d=3$. Let
		$\boldsymbol{u}_h$ be the discrete velocity pair arising from the
		solution of the space-time HDG scheme \cref{eq:discrete_problem}.
		Then
		$\norm[1]{\mathcal{D}^{k_t}_t(u_h)}_{L^{4/d}(0,T;V')} \le
		C(f,u_0,\nu,T)$.
	\end{theorem}
	\begin{proof}
		We follow the strategy used in the proof of \cite[Theorem 3.2]{Li:2015}.
		The density of $C_c(0,T;V)$ in $L^p(0,T;V)$ for $1 \le p < \infty$,  gives us also the density of $\mathscr{M}$ in $L^{4/(4-d)}(0;T,V)$.
		We therefore replace the supremum over $v \in L^{4/(4-d)}(0;T,V)$ in \cref{eq:norm_duality_prod} with the supremum
		over $\varphi \in \mathscr{M}$.
		Now let $\varphi \in \mathscr{M}$ be arbitrary.  Using the
		expansion of $\varphi$ \cref{eq:tensor_product_functions} the
		definitions of the $L^2$-projections $\Pi^t$ and $\Pi_h$, 
		\Cref{prop:proj_bnd}, and \Cref{lem:disc_time_der_bnd_disc_dual}, we have
		\begin{equation*}
			\begin{split}
				&\norm[0]{ \mathcal{D}_t^{k_t}(u_h)}_{L^{4/d}(0,T;V')}  
				\\ & \qquad= \sup_{0\ne \varphi \in \mathscr{M}} \frac{\big|\int_0^T ( \mathcal{D}_t^{k_t}(u_h),\varphi)_{\mathcal{T}_h} \dif t\big|}{\norm{\varphi}_{L^{4/(4-d)}(0,T;V)}} \\
				& \qquad= \sup_{0\ne \varphi \in \mathscr{M}} \frac{\big|\int_0^T ( \mathcal{D}_t^{k_t}(u_h), \Pi \varphi)_{\mathcal{T}_h} \dif t\big|}{ \del{ \int_0^T \tnorm{(\Pi\varphi, \bar \Pi \varphi)}^{4/(4-d)}_v \dif t}^{(4-d)/4}} \del{\frac{ \int_0^T \tnorm{(\Pi\varphi, \bar \Pi \varphi)}^{4/(4-d)}_v \dif t}{\int_0^T \norm{\varphi}_{V}^{4/(4-d)} \dif t}}^{(4-d)/4}\\
				& \qquad \lesssim \sup_{ 0\ne \varphi \in \mathscr{M}} \frac{\big|\int_0^T ( \mathcal{D}_t^{k_t}(u_h), \Pi \varphi)_{\mathcal{T}_h} \dif t\big|}{ \del{ \int_0^T \tnorm{(\Pi \varphi, \bar \Pi \varphi)}^{4/(4-d)}_v \dif t}^{(4-d)/4}}  \\
				& \qquad \lesssim \sup_{ 0\ne\boldsymbol{v}_h\in \mathcal{V}_h^{\text{div}} \times \bar{\mathcal{V}}_h} \frac{\big|\int_0^T ( \mathcal{D}_t^{k_t}(u_h), v_h)_{\mathcal{T}_h} \dif t\big|}{ \del{ \int_0^T \tnorm{\boldsymbol{v}_h}^{4/(4-d)}_v \dif t}^{(4-d)/4}} \le C(f,u_0,\nu,T).
			\end{split}
		\end{equation*}
	\end{proof}
	
	\subsection{Compactness}
	\label{ss:compactness}
	We end this section by summarizing the significance of the
	uniform bounds on the discrete velocity collected
	in \Cref{lem:energy_stability}, 
	\Cref{thm:bnd_disc_grad}, and \Cref{thm:disc_time_der_bnd}.
	In particular, we conclude that subsequences of the discrete velocity solution
	computed by the space-time HDG scheme \cref{eq:discrete_problem}
	converge to a limit function $u$ in suitable topologies. The goal of \Cref{s:convergence} will
	be to show that $u$ is in fact a weak solution to the Navier--Stokes problem in
	the sense of \Cref{def:equiv_weak_form}.
	
	\begin{theorem} \label{thm:convergence_of_subsequence}
		Let $\mathcal{H} \subset (0,|\Omega|)$ be a countable set of mesh sizes whose unique accumulation point is $0$.
		Assume that $\cbr[0]{(\mathfrak{T}_h,\mathfrak{F}_h)}_{h \in \mathcal{H}}$ is a sequence of conforming, shape-regular, and quasi-uniform space-time prismatic meshes,
		and suppose that $\tau \to 0$ as $h \to 0$.
		Suppose that $k_s \ge 1$ and $k_t \ge 0$ if $d=2$ and $k_t \in \cbr{0,1}$ if $d=3$, and for each fixed $h \in \mathcal{H}$, let $\boldsymbol{u}_h \in \boldsymbol{\mathcal{V}}_h$ be the discrete velocity pair computed on $(\mathfrak{T}_h,\mathfrak{F}_h)$
		using the space-time HDG scheme \cref{eq:discrete_problem}. Collect each of these solutions in the sequence
		$\cbr{\boldsymbol{u}_h}_{h \in \mathcal{H}}$. Then, there exists a function $u \in L^{\infty}(0,T;H) \cap L^2(0,T;V)$
		with $\frac{\dif u}{\dif t} \in L^{4/d}(0,T;V')$ such that, up to a (not relabeled) subsequence:
		\begin{align*}
			\emph{(i)} \;& u_h \overset{\star}{\rightharpoonup} u && \hspace{-15mm} \text{in } L^{\infty}(0,T;H) \text{ as } h \to 0,  \\
			\emph{(ii)} \; & u_h \to u   &&\hspace{-15mm} \text{in } L^2(0,T;L^2(\Omega)^d) \text{ as } h\to 0, \\
			\emph{(iii)}\;&G_{h}^{k_s}(\boldsymbol{u}_{h,i}) \rightharpoonup \nabla u_i &&\hspace{-15mm} \text{in } L^2(0,T;L^2(\Omega)^d) \text{ as } h \to 0, \\
			\emph{(iv)} \;&\mathcal{D}_{t}^{k_t}(u_{h}) \rightharpoonup \frac{\dif u}{\dif t} &&\hspace{-15mm} \text{in } L^{4/d}(0,T;V') \text{ as } h \to 0. \\
		\end{align*}
	\end{theorem}
	\begin{proof}
		
		\noindent \textbf{(i)} Weak$-\star$ convergence. The existence of a $u$ satisfying (i) is a direct consequence of the uniform $L^{\infty}(0,T;L^2(\Omega)^d)$
		bound in \Cref{lem:energy_stability} and \cite[Corollary 3.30]{Brezis:book}.
		
		
		\textbf{(ii)} Strong convergence. This follows from \Cref{thm:compactness} due to the uniform energy bound in \Cref{lem:energy_stability}, the uniform bound on $\mathcal{D}_{t}^{k_t}(u_{h})$ in \Cref{lem:disc_time_der_bnd_disc_dual} (see also \Cref{rem:alternate_compact} and \cite[Theorem 3.2]{Li:2015}), and the uniqueness of distributional limits.
		
		
		
		\textbf{(iii)} Weak convergence of the discrete gradient. 
		By \Cref{thm:bnd_disc_grad}
		there exists $w \in L^2(0,T;L^2(\Omega)^d)$ such that, upon passage to a subsequence ,
		$G_h^k(\boldsymbol{u}_{h,i}) \rightharpoonup w$ in $L^2(0,T;L^2(\Omega)^d)$ 
		as $h \to 0$. 
		Let $\phi \in C_c^{\infty}(\mathbb{R}^d)^d$ and $\eta \in 
		C_c^{\infty}(0,T)$ be arbitrary and let $\Pi_h$ be the orthogonal
		$L^2$-projection onto $V_h$.
		Extending $u_{h,i}$, $G_h^{k_s}(\boldsymbol{u}_{h,i})$, $R_h^{k_s}(u_{h,i} - \bar{u}_{h,i})$, $u$, and $w$ by 
		zero outside 
		of $\Omega$, and integrating by parts element-wise in space, we have for all $\eta \in C^{\infty}_c(0,T)$ and $\phi \in C_c^{\infty}(\mathbb{R}^d)^d$  that
		\begin{equation} \label{eq:limit_disc_grad_inter1}
			\begin{split}
				\int_{0}^T& \del[2]{\int_{\mathbb{R}^d} G_h^{k_s}(\boldsymbol{u}_{h,i}) 
					\cdot \phi\dif x }\eta\dif t \\
				&= \int_{0}^T \del[2]{ -\int_{\mathbb{R}^d}  u_{h,i} \nabla \cdot \phi \dif x +\sum_{K \in \mathcal{T}_h} \int_{\partial K} (u_{h,i}  - \bar{u}_{h,i} )
					(\phi - \Pi_h \phi) \cdot n\dif s  
				} \eta \dif t,
			\end{split}
		\end{equation}
		where we have used \cref{eq:disc_grad_op,eq:global_spat_lift}, that $\phi$ and $\bar{u}_{h,i}$ are single-valued
		on element boundaries, and that $\bar{u}_{h,i}|_{\partial \Omega} = 0$.
		Moreover, 
		%
		\begin{equation} \label{eq:limit_disc_grad_inter2}
			\int_0^T\norm{\eta(\phi - \Pi_h \phi) \cdot n}^2_{L^2(\partial K)} \dif t\lesssim h^{2\ell + 1} \int_0^T \eta^2\norm{\phi}^2_{H^{\ell + 1}(\Omega)}\dif t.
		\end{equation}
		As a consequence of \cref{eq:limit_disc_grad_inter1,eq:limit_disc_grad_inter2} and the strong convergence in $L^2(0,T;L^2(\Omega)^d)$ of $u_h$ to $u$, it holds for all $\eta \in C^{\infty}_c(0,T)$ that
		\begin{equation}
			\int_{0}^T \del[2]{\int_{\mathbb{R}^d} w
				\cdot \phi\dif x }\eta\dif t  =	\lim_{h \to 0}\int_{0}^T \del[2]{\int_{\mathbb{R}^d} G_h^{k_s}(\boldsymbol{u}_{h,i}) 
				\cdot \phi\dif x }\eta\dif t = \int_{0}^T \del[2]{ -\int_{\mathbb{R}^d}  u_{i} \nabla \cdot \phi \dif x} \eta \dif t.
		\end{equation}
		Hence $w = \nabla u_i$ as elements of $L^2(0,T;L^2(\mathbb{R}^d)^d)$, so $u_i \in 
		L^2(0,T;H^1(\mathbb{R}^d))$. As $u_i$ vanishes outside of $\Omega$, the $H^1(\mathbb{R}^d)$-regularity ensures that $u_i$ vanishes on the boundary.
		As $u \in H$, its distributional divergence vanishes, and thus $u \in L^2(0,T;V)$. 
		
		\textbf{(iv)} Weak convergence of the discrete time derivative.
		By \Cref{thm:disc_time_der_bnd}, there exists a $z \in L^{4/d}(0,T;V')$ 
		such that, 
		upon passage to a subsequence , $\mathcal{D}_{t}^{k_t}(u_{h}) \rightharpoonup z$ in $L^{4/d}(0,T;V')$.
		%
		%
		For arbitrary $v \in V$ and $\eta \in C_c^{\infty}(0,T)$, we 
		use the definition of $\mathcal{D}_{t}^{k_t}(u_h)$ \cref{eq:discrete_time_der} and integrate by parts in time to find
		\begin{equation} \label{eq:time_der_conv_inter}
			\begin{split}
				\int_0^T & \langle \mathcal{D}_{t}^{k_t}(u_h), v \eta \rangle_{V' \times V} \dif t
				\\
				&= -\int_0^T ( u_h, v  )_{\mathcal{T}_h} \partial_t \eta \dif t + \sum_{n=0}^{N-1} \del{(u_{n+1}^-,v)_{\mathcal{T}_h}  \eta(t_{n+1})  - (u_{n}^-,v )_{\mathcal{T}_h} \eta(t_n) }.
			\end{split}
		\end{equation}
		The telescoping sum on the right hand side of \cref{eq:time_der_conv_inter} vanishes since
		$\eta(0) = \eta(T) = 0$. Thus, we can take the limit as $h \to 0$ to find
		that for all $\eta \in C_c^{\infty}(0,T)$,
		\begin{equation*}
			\begin{split}
				\int_0^T  \eta \langle z, v \rangle_{V' \times V}  \dif t =\lim_{h \to 0} \int_0^T  \langle \mathcal{D}_{t}^{k_t}(u_h), v \eta \rangle_{V' \times V} \dif t 
				= - \int_0^T \partial_t \eta  (u, v )_{\mathcal{T}_h} \dif t,
			\end{split}
		\end{equation*}
		since $\mathcal{D}_{t}^{k_t}(u_{h}) \rightharpoonup z$ in $L^{4/d}(0,T;V')$ and $u_h \to u$ in $L^2(0,T;L^2(\Omega)^d)$ as $h \to 0$. 
		Therefore, $z= \frac{du}{dt}$. 
	\end{proof}
	
	\section{Convergence to weak solutions}
	\label{s:convergence}
	
	The remainder of this article is dedicated to showing that the limiting function $u \in L^{\infty}(0;T,H) \cap L^2(0,T;V)$
	guaranteed by \Cref{thm:convergence_of_subsequence} is actually a weak solution of the Navier--Stokes problem in the sense of \Cref{def:equiv_weak_form}.
	The plan is as follows: we first construct a set of test functions in the discrete space that will allow us to conclude upon
	passage to the limit that $u$ solves \cref{eq:ns_weak_formulation2}. We will then show that
	the viscous term $a_h(\cdot,\cdot)$ and the nonlinear convection term $o_h(\cdot;\cdot,\cdot)$
	enjoy \emph{asymptotic consistency} in the sense of \cite[Definition 5.9]{Pietro:book},
	and use this to pass to the limit in \cref{eq:scheme_rewritten}. Finally, we discuss
	the energy (in)equality and conclude that
	the constructed weak solution $u \in L^{\infty}(0,T;H) \cap L^2(0,T;V)$ is a solution in the sense of Leray--Hopf. 
	
	From
	this point on, we assume that $\mathcal{H} \subset (0,|\Omega|)$ is a countable set of mesh sizes whose unique accumulation point is $0$,
	$\cbr[0]{(\mathfrak{T}_h,\mathfrak{F}_h)}_{h \in \mathcal{H}}$ is a sequence of conforming and quasi-uniform space-time prismatic meshes,
	and $\tau \to 0$ as $h \to 0$. For each fixed $h \in \mathcal{H}$, we let $\boldsymbol{u}_h = (u_h,\bar{u}_h) \in \mathcal{V}_h^{\text{div}} \times \bar{\mathcal{V}}_h$ be a discrete velocity pair computed on $(\mathfrak{T}_h,\mathfrak{F}_h)$
	using the space-time HDG scheme \cref{eq:discrete_problem} for $n=0,\dots,N-1$. We collect each of these solutions in the sequence
	$\cbr{\boldsymbol{u}_h}_{h \in \mathcal{H}}$.
	
	\subsection{Strong convergence of test functions}
	
	Passing to the limit in \cref{eq:scheme_rewritten} will require a suitable set of discrete test functions. 
	We will again use the set $\mathscr{M}$ of functions defined in \Cref{lem:tensor_product_functions} 
	as our basic building block, as it is sufficiently rich to ensure density in $C_c(0,T;V)$ while 
	its tensor-product structure allows us to easily combine
	spatial and temporal projections onto the discrete spaces. In particular, given $\varphi \in \mathscr{M}$,
	we will work with the discrete functions $\Pi \varphi$ and $\bar \Pi \varphi$, constructed in \cref{eq:test_func_for_time_der_bnd}.
	To set notation, we denote by $\Pi \varphi_i$ and $\bar \Pi \varphi_i$ the $i$th Cartesian
	component of the vector functions $\Pi \varphi$ and $\bar \Pi \varphi$, respectively. We first
	show a strong convergence result for the sequence of discrete 
	test functions $\cbr[0]{(\Pi \varphi, \bar{\Pi} \varphi)}_{h \in \mathcal{H}}$:
	\begin{proposition}\label{prop:test_func}
		Let $k_s \ge 1$,  $k_t \ge 0$ if $d=2$, and $k_t \in \cbr{0,1}$ if $d=3$. Suppose that $\varphi \in \mathscr{M}$
		and consider the sequence of discrete test functions $\cbr[0]{(\Pi \varphi, \bar{\Pi} \varphi)}_{h \in \mathcal{H}}$ defined
		in \cref{eq:test_func_for_time_der_bnd}.
		Then, it holds that $\Pi \varphi \to \varphi$ strongly in $L^{\infty}(0,T;L^{\infty}(\Omega)^d)$ and $G_{h}^{k_s}(( \Pi \varphi_i, \bar \Pi \varphi_i)) \to \nabla \varphi_i$ strongly in $L^2(0,T;L^2(\Omega)^d)$ as 
		$h\to 0$.
	\end{proposition}
	
	\begin{proof}
		We first record the following consequences of \Cref{lem:approximation_div}:
		\begin{subequations}
			\label{eq:timeprojest_discgradconvinter2}
			\begin{align}
				\label{eq:time_proj_est}
				\norm[0]{\Pi^{t}\eta_k - \eta_k}_{L^{\infty}(0,T)}^2 &\lesssim \tau^2 \norm{\eta_k}_{W^{1,\infty}(0,T)}^2, 
				\\  
				\sum_{K \in \mathcal{T}_h}\norm[0]{ \nabla (\Pi_h^{\text{div}} \psi_k -  \psi_k)}_{L^2(K)}^2 &\lesssim h^2 \norm{\psi_k}_{H^2(\Omega)}^2,
				\\
				\norm[0]{\psi_k - \Pi_h^{\text{div}} \psi_k}_{L^{\infty}(\Omega)} &\lesssim h^{1/2} |\psi_k|_{H^{2}(\Omega)}.  \label{eq:div_inf_est}
			\end{align}    
		\end{subequations}
		That $\Pi \varphi \to \varphi$ strongly in $L^{\infty}(0,T;L^{\infty}(\Omega)^d)$ follows from \cref{eq:timeprojest_discgradconvinter2}, since
		\begin{equation}
			\begin{split}
				&\norm[0]{\varphi - \Pi \varphi}_{L^{\infty}(0,T;L^{\infty}(\Omega)^d)} \\
				& \lesssim \sum_{k=1}^m\del{\norm[0]{ \eta_k}_{L^{\infty}(0,T)} \norm[0]{\psi_k - \Pi_h^{\text{div}} \psi_k}_{L^{\infty}(\Omega)} + \norm[0]{\eta_k -\Pi^t \eta_k}_{L^{\infty}(0,T)} \norm[0]{\psi_k}_{H^2(\Omega)}}.
			\end{split}
		\end{equation} 
		We now prove the strong convergence of $G_{h}^{k_s}(( \Pi \varphi_i, \bar \Pi \varphi_i))$ to $\nabla \varphi_i$
		in $L^2(0,T;L^2(\Omega)^d)$. Using the definition of the discrete gradient \cref{eq:disc_grad_op}, the
		triangle inequality, and \cref{eq:bnd_on_lift},
		we find
		\begin{equation} \label{eq:test_disc_grad_bnd}
			\begin{split}
				\int_0^T&\norm[0]{G_h^{k_s}((\Pi \varphi_i, \bar{\Pi} \varphi_i)) - \nabla 
					\varphi_i}_{L^2(\Omega)}^2 \dif t \\
				& \le \sum_{K \in \mathcal{T}_h} \int_0^T \norm{\nabla \Pi_h \varphi - \nabla \varphi}_{L^2(K)}^2\dif t +\sum_{K \in \mathcal{T}_h} \int_0^T h_K^{-1} \norm{ \Pi \varphi - \bar \Pi \varphi}_{L^2(\partial K)}^2 \dif t.
			\end{split}
		\end{equation}
		We start with the first term on the right hand side of \cref{eq:test_disc_grad_bnd}.
		By the definition of
		$\Pi \varphi$, the triangle inequality, and \cref{eq:timeprojest_discgradconvinter2}, we can write
		\begin{equation*}
			\begin{split} 
				\sum_{K \in \mathcal{T}_h} &\int_0^T \norm{\nabla \Pi \varphi - \nabla \varphi}_{L^2(K)}^2 \dif t
				\\ & \lesssim \sum_{k=1}^M \sum_{K \in \mathcal{T}_h} \del[3]{\int_0^T (\Pi^{t} \eta_k)^2 \norm[0]{ \nabla (\Pi_h^{\text{div}} \psi_k -  \psi_k)}_{L^2(K)}^2 \dif t + \int_0^T( \Pi^{t} \eta_k  - \eta_k)^2\norm[0]{ \nabla \psi_k}_{L^2(K)}^2 \dif t }
				\\ &\lesssim \sum_{k=1}^M  \norm{\eta_k}_{W^{1,\infty}(0,T)}^2 \del[3]{  h^2  \int_0^T \norm{\psi_k}_{H^2(\Omega)}^2 \dif t + \tau^2  \int_0^T \norm{\psi_k}_{L^2(\Omega)}^2 \dif t},
			\end{split}
		\end{equation*}
		which can be seen to vanish as $h\to 0$.
		Turning now to the second term on the right hand side of \cref{eq:test_disc_grad_bnd}, we find
		\begin{equation} \label{eq:disc_grad_conv_inter3}
			\begin{split}
				\sum_{K \in \mathcal{T}_h} &\int_0^T h_K^{-1} \norm{ \Pi \varphi - \bar \Pi \varphi}_{L^2(\partial K)}^2 \dif t \\& \lesssim \sum_{k=1}^M
				\norm{\Pi^{t} \eta_k}^2_{W^{1,\infty}(0,T)} \sum_{K \in \mathcal{T}_h} \int_0^T h_K^{-1}  \norm[0]{ \Pi_h^{\text{div}} \psi_k - \bar{\Pi}_h \psi_k}_{L^2(\partial K)}^2 \dif t.
			\end{split}
		\end{equation}
		Using a discrete local trace inequality, the assumed quasi-uniformity of the space-time prismatic mesh, and the approximation properties of $\Pi_h^{\text{div}}$ and $\bar{\Pi}_h$, we find
		\begin{equation}  \label{eq:div_facet_est}
			h_K^{-1}\norm[0]{ \Pi_h^{\text{div}} \psi_k - \bar{\Pi}_h \psi_k}_{L^2(\partial K)}^2 \lesssim h^2\norm{\psi_k}_{H^2(\Omega)}^2,
		\end{equation}
		and thus the right hand side of \cref{eq:disc_grad_conv_inter3} vanishes as $h \to 0$. The result follows.
	\end{proof}

	\subsection{Asymptotic consistency of the linear viscous term}
	
	We are now in a position to show that the linear viscous term
	is asymptotically consistent in the sense of \cite[Definition 5.9]{Pietro:book}:
	
	\begin{theorem} \label{thm:asym_consist_visc}
		Let $k_s \ge 1$, $k_t \ge 0$ if $d=2$, and $k_t \in \cbr{0,1}$ if $d=3$. Let $\varphi \in \mathscr{M}$, denote by $(\Pi \varphi, \bar \Pi \varphi)$
		the discrete test functions constructed as in \cref{eq:test_func_for_time_der_bnd}, and let
		$u \in L^{\infty}(0,T;H) \cap L^2(0,T;V)$ be the limit (up to a subsequence)
		of $\cbr{u_h}_{h \in \mathcal{H}}$ guaranteed by \Cref{thm:convergence_of_subsequence}.
		Then, the following asymptotic consistency result holds for the linear
		viscous term:
		\begin{equation*}
			\lim_{h \to 0} 	\int_{0}^T a_h(\boldsymbol{u}_h, (\Pi \varphi, 
			\bar \Pi \varphi)) 
			\dif t =  \int_0^T \int_{\Omega} \nabla u : \nabla \varphi
			\dif x \dif t.
		\end{equation*}
	\end{theorem}
	
	\begin{proof}
		Since
		$G_h^{k_s}((\Pi \varphi_i, \bar \Pi \varphi_i)) \to \nabla
		\varphi_i$ strongly in $L^2(0,T;L^2(\Omega)^d)$ by
		\Cref{prop:test_func} and by \Cref{thm:convergence_of_subsequence}
		(iii), $G_h^{k_s}(\boldsymbol{u}_{h,i}) \to \nabla u_i$ weakly in
		$L^2(0,T;L^2(\Omega)^d)$ as $h\to 0$, we can pass to the limit in
		the first term of \cref{eq:viscous_rewritten} to find that
		\begin{equation*}
			\lim_{h \to 0} \int_0^T \int_{\Omega} 
			G_h^{k_s}(\boldsymbol{u}_{h,i}) 
			\cdot 
			G_h^{k_s}( (\Pi \varphi_i, 
			\bar \Pi \varphi_i)) \dif x \dif t = \int_0^T \int_{\Omega}\nabla 
			u_i \cdot 
			\nabla \varphi_i \dif x \dif t.
		\end{equation*}
		Turning to the second term of \cref{eq:viscous_rewritten},
		we have by the Cauchy-Schwarz inequality, the
		bound on the global spatial lifting operator \cref{eq:bnd_on_lift},
		the definition of $\Pi \varphi$ and $\bar \Pi \varphi$, and
		uniform bound \Cref{lem:energy_stability},
		\begin{equation*}
			\begin{split}
				\int_0^T&\int_{\Omega} R_h^{k_s}(u_{h,i}-\bar{u}_{h,i}) 
				\cdot  
				R_h^{k_s}(\Pi \varphi_i -\bar{\Pi} \varphi_i)  \dif x \dif 
				t  \\
				& \le C(f,u_0,\nu)
				\del[3]{\sum_{k=1}^M \sum_{K \in \mathcal{T}_h} \int_0^T h_K^{-1} \Pi^{t} \eta_k
					\norm{\Pi_h^{\text{div}} \psi_k -\bar{\Pi}_h \psi_k}^2_{L^2(\partial K)} }^{1/2},
			\end{split} 
		\end{equation*}
		which can be seen to vanish as $h\to 0$ by \cref{eq:time_proj_est} and \cref{eq:div_facet_est}.
		In an identical fashion, we find
		\begin{equation*}
			\lim_{h \to 0} \sum_{K\in\mathcal{T}_h} \int_0^T\int_{\partial 
				K}\frac{\alpha  
			}{h_K}(u_h - \bar{u}_h)\cdot(\Pi \varphi_i-\bar \Pi \varphi_i) \dif s \dif t = 0.
		\end{equation*}
		The result follows.
	\end{proof}

	\subsection{Asymptotic consistency of the nonlinear convection term}
	We now prove that the nonlinear convection term
	is asymptotically consistent in the sense of \cite[Definition 5.9]{Pietro:book}:

	\begin{theorem} \label{thm:asym_consist_conv}
		Suppose that $k_s \ge 1$ and $k_t \ge 0$ if $d=2$ and $k_t \in \cbr{0,1}$ if $d=3$. Let $\varphi \in \mathscr{M}$, denote by $(\Pi \varphi, \bar \Pi \varphi)$
		the discrete test functions constructed as in \cref{eq:test_func_for_time_der_bnd}, and let
		$u \in L^{\infty}(0,T;H) \cap L^2(0,T;V)$ be an accumulation point
		of $\cbr{u_h}_{h \in \mathcal{H}}$ guaranteed by \Cref{thm:convergence_of_subsequence}.
		Then, the following \emph{asymptotic consistency} result holds for the nonlinear
		convection term:
		\begin{equation*}
			\lim_{h \to 0} 	\int_{0}^T o_h(u_h; \boldsymbol{u}_h, (\Pi \varphi, 
			\bar{\Pi} \varphi)) 
			\dif t =  \int_0^T \int_{\Omega} (u \cdot \nabla u) \cdot 
			\varphi
			\dif x \dif t.
		\end{equation*}
		
	\end{theorem}
	
	\begin{proof}
		We start with the first term on the right hand side of \cref{eq:convection_rewritten}. 
		By H\"{o}lder's inequality, we have
		\begin{equation*}
			\begin{split}
				&\int_0^T \norm{u \varphi_i -u_h \Pi \varphi_i}_{L^2(\Omega)}^2\dif t\\
				&\lesssim \norm[0]{\varphi - \Pi \varphi}_{L^{\infty}(0,T;L^{\infty}(\Omega)^d)}^2  \int_0^T \norm{u}_{L^2(\Omega)}^2 \dif t+ \norm[0]{\Pi \varphi}_{L^{\infty}(0,T;L^{\infty}(\Omega)^d)}^2 \int_0^T \norm{u -u_h }_{L^2(\Omega)}^2\dif t,
			\end{split}
		\end{equation*}
		which can be seen to vanish as $h\to0$ by \Cref{prop:test_func} and \Cref{thm:convergence_of_subsequence}.
		Therefore, $u_h \Pi \varphi_i \to u\varphi_i $ strongly in $L^2(0,T;L^2(\Omega)^d)$ as $h\to 0$,
		and this combined with the fact that $G_h^{2k_s}(\boldsymbol{u}_{h,i}) \rightharpoonup \nabla u$ yields
		\begin{equation*}
			\lim_{h \to 0} \int_0^T \int_{\Omega} u_h \cdot 
			G_h^{2k_s}(\boldsymbol{u}_{h,i}) 
			\Pi \varphi_i
			\dif x \dif t = \int_0^T \int_{\Omega} (u \cdot \nabla u) \cdot 
			\varphi
			\dif x \dif t.
		\end{equation*}
		It remains to show that the facet term appearing in \cref{eq:convection_rewritten}
		converges to $0$ as $h \to 0$.
		By the definitions of $\Pi \varphi$ and $\bar{\Pi}
		\varphi$, proceeding as in the proof of \cite[Proposition 3.4]{Cesmelioglu:2017}, 
		and using the fact that $ \int_{0}^T  
		\tnorm{\boldsymbol{u}_h}_v^2 
		\dif t$ is uniformly bounded by \Cref{lem:energy_stability},
		\begin{equation*}
			\begin{split}
				\bigg|\sum_{K\in\mathcal{T}_h}& \int_{0}^T\int_{\partial 
					K}\tfrac{1}{2}\del{u_h 
					\cdot n + 
					\envert{u_h\cdot 
						n}}(u_h-\bar{u}_h)\cdot(\Pi \varphi -
				\bar{\Pi} \varphi)\dif s \dif t \bigg| \\
				& \le C(f,u_0,\nu) \sum_{k=1}^M \norm[0]{ \Pi^{t}
					\eta_k}_{L^{\infty}(0,T)}  
				\del[3]{\sum_{K\in\mathcal{T}_h}h_K^{-1}
					\norm[0]{\Pi_h^{\text{div}} 
						\psi_k -\bar{\Pi}_h 
						\psi_k}_{L^2(\partial K)}^2}^{1/2},
			\end{split}
		\end{equation*}
		which can be seen to vanish as $h \to 0$ by using the second bound in \cref{eq:time_proj_est} and  \cref{eq:div_facet_est}.
	\end{proof}
	
	\subsection{Passing to the limit} \label{ss:pass_to_limit}
	
	With the asymptotic consistency of the linear viscous term  (\Cref{thm:asym_consist_visc}) and the nonlinear
	convection term (\Cref{thm:asym_consist_conv}), we are ready to pass to the limit
	as $h \to 0$ in \cref{eq:scheme_rewritten}. 
	Extract from $\cbr{\boldsymbol{u}_h}_{h \in \mathcal{H}}$ the subsequence satisfying the convergence results
	listed in \Cref{thm:convergence_of_subsequence}.
	Let $\varphi \in \mathscr{M}$ and
	choose $\boldsymbol{v}_h = (\Pi \varphi, \bar \Pi \varphi) \in \mathcal{V}_{h}^{\text{div}} \times \bar{\mathcal{V}}_h$
	as a test function in \cref{eq:scheme_rewritten}:
	\begin{equation} \label{eq:pre_pass_limit}
		\begin{split}
			\int_{0}^T &( \mathcal{D}_t^{k_t} (u_h),\Pi \varphi)_{\mathcal{T}_h} \dif t  +
			\int_{0}^T \del{ \nu a_h(\boldsymbol{u}_h, (\Pi\varphi, \bar \Pi \varphi)) + o_h(u_h; 
				\boldsymbol{u}_h, (\Pi \varphi, \bar \Pi \varphi) } \dif t =   \int_{0}^T (f, 
			\Pi \varphi)_{\mathcal{T}_h} 
			\dif t.
		\end{split}
	\end{equation}
	By the definition of $\Pi \varphi$, we have 
	\begin{equation}
		\int_{0}^T ( \mathcal{D}_t^{k_t} (u_h),\Pi \varphi)_{\mathcal{T}_h} \dif t = \int_{0}^T( \mathcal{D}_t^{k_t} (u_h),\varphi)_{\mathcal{T}_h} \dif t = \int_{0}^T \langle \mathcal{D}_t^{k_t} (u_h),\varphi \rangle_{V' \times V} \dif t.
	\end{equation}
	Thus, the weak convergence of $\mathcal{D}_t^{k_t} (u_h)$ to $\od{u}{t}$ in $L^{4/d}(0,T;V')$
	yields
	\begin{equation}
		\lim_{h \to 0} \int_0^T( \mathcal{D}_t^{k_t} (u_h),\Pi \varphi)_{\mathcal{T}_h}  \dif t = \int_{0}^T \big \langle \od{u}{t}, \varphi \big\rangle_{V'\times V}\dif t.
	\end{equation}

	This, in combination with \Cref{thm:asym_consist_visc} and \Cref{thm:asym_consist_conv},
	shows that upon passage to the limit as $h \to 0$ in \cref{eq:pre_pass_limit} that the
	limit $u \in L^{\infty}(0,T;H) \cap L^2(0,T;V)$ of
	the subsequence $\cbr{u_h}_{h \in \mathcal{H}}$ given by \Cref{thm:convergence_of_subsequence}
	satisfies for all $\varphi \in \mathscr{M}$, 
	\begin{equation}  \label{eq:lim_pass_pre_density}
		\int_0^T \bigg \langle \od{u}{t}, \varphi \bigg\rangle_{V' \times V} 
		\dif t 
		+ 
		\int_0^T((u \cdot \nabla)u , \varphi) \dif t + \nu \int_0^T (\nabla u, 
		\nabla 
		\varphi) \dif t = \int_0^T ( f, \varphi)_{\mathcal{T}_h}
		\dif t.
	\end{equation}
	By the density of the set $\mathscr{M}$ in $C_c(0,T;V)$, \cref{eq:lim_pass_pre_density}
	holds also for all $\varphi \in C_c(0,T;V)$. 
	
	We now show that $u \in L^{\infty}(0,T;H) \cap L^2(0,T;V)$ satisfies the initial condition in the sense that
	$u(0) = u_0$ in $V'$.
	Our starting point is
	the definition of the discrete time derivative \cref{eq:discrete_time_der} 
	in a single space-time slab $\mathcal{E}^n$:
	\begin{equation} \label{eq:init_cond_inter}
		\int_{I_n}(\partial_t u_h, v_h )_{\mathcal{T}_h}\dif t +(u_n^+ - u_n^-,v_n^+)_{\mathcal{T}_h}= \int_{I_n} (\mathcal{D}_{t}^{k_t}(u_h), v_h )_{\mathcal{T}_h} \dif t, \quad \forall v_h \in \mathcal{V}_h.
	\end{equation}
	Let $\psi \in V$ and $\eta \in C^{\infty}(0,T)$ such that
	$\eta(T) = 0$.  Define a function $w_h \in \mathcal{V}_h^{\text{div}}$
	by setting
	$w_h|_{\mathcal{E}^n} = \Pi_n^{k_t} \eta\Pi_h^{\text{div}} \psi$ with
	$\Pi_n^{k_t}: H^1(I_n) \to P_{k_t}(I_n)$ the DG time projection
	defined as in \cite[Section 69.3.2]{Ern:bookiii}. Define also the global projection
	$\Pi^{k_t}|_{I_n} = \Pi^{k_t}_n$.
	We note that by definition, $(\psi, v_h)_{\mathcal{T}_h} = (\Pi_h^{\text{div}} \psi, v_h)_{\mathcal{T}_h}$
	for all $v_h \in V_h^{\text{div}}$, and by the defining properties of the projection $\Pi_n^{k_t}$ (see \cite[Eq. (69.26)]{Ern:bookiii}), 
	\begin{equation*}
		\int_{I_n}(\partial_t u_h, \Pi_h^{\text{div}} \psi )_{\mathcal{T}_h} \Pi_{n}^{k_t} \eta \dif t
		= \int_{I_n}(\partial_t u_h, \Pi_h^{\text{div}} \psi )_{\mathcal{T}_h} \eta \dif t
		\quad \text{and} \quad
		(\Pi_n^{k_t}\eta)(t_n^+) = \eta(t_n),
	\end{equation*}
	since $\partial_t u_h \in P_{k_t -1}(I_n)$ (with the convention that $P_{-1}(I_n) = \cbr{0}$).
  Choosing $v_h = w_h$
	in \cref{eq:init_cond_inter}, we find
	\begin{equation} \label{eq:init_cond_IBP_1}
		\int_{I_n}(\partial_t u_h, \psi )_{\mathcal{T}_h} \eta \dif t + (u_n^+ - u_n^-,\psi )_{\mathcal{T}_h} \eta(t_n) = \int_{I_n} (\mathcal{D}_{t}^{k_t}(u_h), \psi )_{\mathcal{T}_h} \Pi^{k_t} \eta \dif t.
	\end{equation}
	%
	%
	Integrating by parts in time on the left hand side of \cref{eq:init_cond_IBP_1}, summing
	over all space-time slabs, and using that $\eta(T) = 0$, we have
	%
	\begin{equation} \label{eq:init_cond_inter3}
		-\int_{0}^T(u_h, \psi )_{\mathcal{T}_h} \partial_t \eta \dif t  - ( u_0^-,\psi )_{\mathcal{T}_h} \eta(0) = \int_{0}^T (\mathcal{D}_{t}^{k_t}(u_h), \psi )_{\mathcal{T}_h} \Pi^{k_t} \eta \dif t.
	\end{equation}
	%
	
	From \Cref{thm:convergence_of_subsequence} (i) and (iii), and since $u_0^- = \Pi_h^{\text{div}} u_0 \to u_0$ strongly in $H$, and $\Pi^{k_t} \eta \to \eta$ strongly in $L^{4/(4-d)}(0,T)$ by \cref{eq:time_proj_est}, we can pass to the limit as $h \to 0$ in \cref{eq:init_cond_inter3}
	to find
	\begin{equation} \label{eq:init_cond_inter4}
		-\int_{0}^T(u, \psi )_{\mathcal{T}_h}  \partial_t \eta \dif t  - ( u_0,\psi )_{\mathcal{T}_h} \eta(0) = 	\int_{0}^{T} \bigg \langle \od{u}{t}, \psi \bigg \rangle_{V' \times V} \eta\dif t.
	\end{equation}
	Comparing \cref{eq:init_cond_inter4} with \cite[Theorem II.5.12]{Boyer:book}, we find that
	%
	%
	\begin{equation*}
		0 = (u(0) - u_0, \psi)_{\mathcal{T}_h} = \langle u(0) - u_0, \psi \rangle_{V' \times V}, \; \forall \psi \in V \quad \Rightarrow \quad u(0) = u_0 \text{ in } V'.
	\end{equation*}
	Therefore, we have proven:
	\begin{theorem} \label{thm:convergence}
		Let $u_0 \in H$ and $f \in L^2(0,T;L^2(\Omega)^d)$ be given and let $\mathcal{H} \subset (0,|\Omega|)$ be a countable set of mesh sizes whose unique accumulation point is $0$.
		Assume that $\cbr[0]{(\mathfrak{T}_h,\mathfrak{F}_h)}_{h \in \mathcal{H}}$ is a sequence of conforming and quasi-uniform space-time prismatic meshes,
		and suppose that $\tau \to 0$ as $h \to 0$.
		Suppose that $k_s \ge 1$, $k_t \ge 0$ if $d=2$ and $k_t \in \cbr{0,1}$ if $d=3$, and for each fixed $h \in \mathcal{H}$, let $\boldsymbol{u}_h \in \boldsymbol{\mathcal{V}}_h$ be the discrete velocity pair computed on $(\mathfrak{T}_h,\mathfrak{F}_h)$
		using the space-time HDG scheme \cref{eq:discrete_problem} for $n=0,\dots, N-1$. Collect these discrete solutions in the sequence
		$\cbr{\boldsymbol{u}_h}_{h \in \mathcal{H}}$. Then, upon passage to a subsequence, $\cbr{u_h}_{h \in \mathcal{H}}$ converges as $h\to 0$ (in the sense of \Cref{thm:convergence_of_subsequence}) to
		a weak solution of the Navier--Stokes problem \cref{eq:ns_weak_formulation2} $u \in L^{\infty}(0,T;H) \cap L^2(0,T;V)$ with
		$\od{u}{t} \in L^{4/d}(0,T;V')$.
	\end{theorem}
	
	\subsection{The energy inequality}
	%
	
	In three dimensions, we are not guaranteed uniqueness of weak solutions
	and cannot conclude a priori that the weak solution obtained from \Cref{thm:convergence} 
	satisfies the energy inequality \cref{eq:energy_ineq}. We show below that the week solution in fact does satisfy \cref{eq:energy_ineq}.
	
	\begin{lemma}
		Let $d = 3$, $k_s \ge 1$, and $k_t \in \cbr{0,1}$.
		The weak solution
		$u \in L^{\infty}(0,T;H) \cap L^2(0,T;V)$ given by \Cref{thm:convergence} satisfies
		the energy inequality for a.e. $s \in (0,T]$:
		\begin{equation}
			\norm{u(s)}_{L^2(\Omega)}^2 + 2\nu \int_0^s \norm{u}_V^2 \dif t \; \le \;	\norm{u_0}_{L^2(\Omega)}^2 + 2 \int_0^s ( f,u)_{L^2(\Omega)}\dif t.
		\end{equation}
		%
	\end{lemma}
	\begin{proof}
		Let $s \in (0,T)$ be fixed and choose $n_s \in \cbr{0, 1, \dots, N-1}$
		such that $t_{n_{s}} \le s \le t_{n_{s}+1}$. Testing \cref{eq:discrete_problem}
		with $\boldsymbol{v}_h = \boldsymbol{u}_h \in \mathcal{V}_h^{\text{div}} \times \bar{\mathcal{V}}_h$,
		using \cref{eq:lower_bnd_a_h} and the stabillity of $\Pi_h^{\text{div}}$ in $L^2(\Omega)^d$, and summing from $n=0$ to $n=n_s$, we have
		%
		\begin{equation} \label{eq:energy_ineq_inter1}
			\norm[1]{u_{n_s+1}^-}_{L^2(\Omega)}^2 +
			2\nu\sum_{i=1}^3 \int_{0}^{s}\norm[0]{G_{h}^{k_s}(\boldsymbol{u}_h)}_{L^2(\Omega)}^2 \dif t  \le 
			\norm[0]{u_0}_{L^2(\Omega)}^2 + 2\int_{0}^{t_{n_s+1}} (f,u_h)_{\mathcal{T}_h} \dif t.
		\end{equation}
		Let us first suppose that $k_t = 0$. Since $u_{t_{n_s+1}}^- = u_h(s)$ for
		$s \in (t_{n_s}, t_{n_{s}+1})$ in this case, \cref{eq:energy_ineq_inter1} yields
		\begin{equation} \label{eq:energy_ineq_inter2}
			\norm[1]{u_h(s)}_{L^2(\Omega)}^2 +
			2\nu\sum_{i=1}^3 \int_{0}^{s}\norm[0]{G_{h}^{k_s}(\boldsymbol{u}_h)}_{L^2(\Omega)}^2 \dif t  \le 
			\norm[0]{u_0}_{L^2(\Omega)}^2 + 2\int_{0}^{t_{n_s+1}} (f,u_h)_{\mathcal{T}_h} \dif t.
		\end{equation}
		Our goal is to justify passage to the limit in \cref{eq:energy_ineq_inter2}. 	%
		
		To this end, we multiply both sides of \cref{eq:energy_ineq_inter2} by an arbitrary $\phi \in C_c^{\infty}(\mathbb{R})$ satisfying $\phi \ge 0$
		and integrate from $s=0$ to $s=T$:
		\begin{equation} \label{eq:energy_ineq_inter3}
			\begin{split}
			\int_{0}^T \bigg( \norm[1]{u_h(s)}_{L^2(\Omega)}^2 &+
			2\nu \sum_{i=1}^3\int_{0}^{s}\norm[0]{G_{h}^{k_s}(\boldsymbol{u}_{h,i})}_{L^2(\Omega)}^2 \dif t \bigg) \phi(s) \dif s \\&\le 
			\int_0^T \bigg(\norm[0]{u_0}_{L^2(\Omega)}^2 + 2\int_{0}^{t_{n_s+1}} (f,u_h)_{\mathcal{T}_h} \dif t\bigg) \phi(s) \dif s.
			\end{split}
		\end{equation}
		%
		%
		%
		We first consider the integral involving the body force $f$.
		By the triangle inequality, the Cauchy-Schwarz inequality, discrete Poincar\'{e} inequality \cref{eq:disc_poinc}, and
		the uniform energy bound in \Cref{lem:energy_stability}, we obtain
		\begin{equation} \label{eq:estimate_f}
			\bigg| \int_{0}^{t_{n_s+1}} (f,u_h)_{\mathcal{T}_h} \dif t - \int_{0}^{s} (f,u)_{\mathcal{T}_h} \dif t \bigg| \lesssim \del[3]{\int_{s}^{s+\tau} \norm{f}_{L^2(\Omega)}^2 \dif t}^{1/2} +  \int_{0}^{s} |(f,u_h -u)_{\mathcal{T}_h}| \dif t .
		\end{equation}
		%
		%
		Since $f \in L^2(0,T;L^2(\Omega)^d)$, the primitive $F(\tau) = \int_{s}^{s+\tau} \norm{f}_{L^2(\Omega)}^2 \dif t$ 
		is absolutely continuous. This, combined with the fact $u_h \to u$ strongly in $L^2(0,T;L^2(\Omega)^d)$, 
		shows that the right hand side of \cref{eq:estimate_f} vanishes as $h \to 0$ (since also $\tau \to 0$),
		and so
		\begin{equation*}
			\lim_{h \to 0} \int_{0}^{t_{n_s+1}} (f,u_h)_{\mathcal{T}_h} \dif t = \int_{0}^{s} (f,u)_{\mathcal{T}_h} \dif t.
		\end{equation*}
		Thus, we can apply Lebesgue's dominated convergence theorem to find
		\begin{equation} \label{eq:dct_source}
			\lim_{h \to 0} \int_0^T \bigg(\int_{0}^{t_{n_s+1}} (f,u_h)_{\mathcal{T}_h} \dif t \bigg) \phi(s) \dif s = \int_0^T \bigg(\int_{0}^{s} (f,u)_{\mathcal{T}_h} \dif t \bigg)\phi(s) \dif s.
		\end{equation}
		With \cref{eq:dct_source} in hand, we  pass to the lower limit as $h \to 0$ in \cref{eq:energy_ineq_inter3} and use
		Fatou's lemma, the weak lower semicontinuity of norms, and \Cref{thm:convergence_of_subsequence}:
		%
		\begin{equation*}
			\begin{split}
				\int_{0}^T& \bigg( \norm[1]{u(s)}_{L^2(\Omega)}^2 +
				2\nu \int_{0}^{s}\norm[0]{u}_{V}^2 \dif t \bigg) \phi(s) \dif s \le 
				\int_0^T \bigg(\norm[0]{u_0}_{L^2(\Omega)}^2 + 2\int_{0}^{s} (f,u)_{\mathcal{T}_h} \dif t\bigg) \phi(s) \dif s.
			\end{split}
		\end{equation*}
		As this holds for all $\phi \in C_c^{\infty}(\mathbb{R})$ satisfying
		$\phi \ge 0$, we have for a.e. $s \in [0,T]$
		\cite[pp. 291]{Temam:book},
		\begin{equation*}
			\norm[1]{u(s)}_{L^2(\Omega)}^2 +
			2\nu \int_{0}^{s}\norm[0]{u}_{V}^2 \dif t  \le 
			\norm[0]{u_0}_{L^2(\Omega)}^2 + 2\int_{0}^{s} (f,u)_{\mathcal{T}_h} \dif t.
		\end{equation*}
		%

		Next, suppose that $k_t = 1$. In this case, \cref{eq:energy_ineq_inter1} 
		does not offer direct control over $\norm[0]{u_h(s)}_{L^2(\Omega)}$ for $s \in (t_{n_s}, t_{n_s+1})$.
		Instead, we define $\tilde{u}_h$ to be piecewise constant (in time) function
		satisfying $\tilde{u}_h|_{\mathcal{E}^m} = u_{m+1}^- = u_h(t_{m+1}^-)$, so that \cref{eq:energy_ineq_inter1} 
		yields
		\begin{equation*}
			\norm[0]{\tilde{u}_h(s)}_{L^2(\Omega)}^2 +
			2\nu\sum_{i=1}^3 \int_{0}^{s}\norm[0]{G_{h}^{k_s}(\boldsymbol{u}_h)}_{L^2(\Omega)}^2 \dif t   \le 
			\norm[0]{u_0^-}_{L^2(\Omega)}^2 + 2\int_{0}^{t_{n_s+1}} (f,u_h)_{\mathcal{T}_h} \dif t.
		\end{equation*}
		Note that if $u_h \rightharpoonup u$ in $L^2(0,T;L^2(\Omega)^d)$ as $h \to 0$, then also $\tilde{u}_h \rightharpoonup u$ 
		in $L^2(0,T;L^2(\Omega)^d)$ as $h \to 0$ \cite[Corollary 3.2]{Walkington:2005}. The weak lower semi-continuity of the norm in $L^2(0,T;L^2(\Omega)^d)$ yields
		\begin{equation*}
			\int_{0}^T  \norm[0]{u(s)}_{L^2(\Omega)}^2 \phi(s) \dif s \le \liminf_{h \to 0}	\int_{0}^T  \norm[0]{\tilde{u}_h(s)}_{L^2(\Omega)}^2 \phi(s) \dif s.
		\end{equation*}
		%
		The remainder of the proof is identical to the case $k_t = 0$.
	\end{proof}

%
	
	\appendix
	\renewcommand{\thesection}{\Alph{section}}
	
	\section{Discrete compactness of the velocity}
	
	In this section, we recall the discrete compactness theory for DG time stepping 
	developed by Walkington in \cite{Walkington:2010} with a minor modification to 
	fit the current non-conforming spatial discretization. 
	
	\begin{remark} \label{rem:alternate_compact}
		To our knowledge, the compactness theorem in \cite[Theorem 3.1]{Walkington:2010} for
		DG-in-time discretizations was first
		extended to non-conforming spatial approximations in \cite{Li:2015}. 
		%
		Note that we can apply \cite[Theorem 3.2]{Li:2015} in our setting to
		conclude that the sequence $\cbr{u_h}_{h \in \mathcal{H}}$ is relatively compact in $L^2(0,T;L^2(\Omega)^d)$ by selecting (using the notation of \cite{Li:2015} with $Y$ and $X$ replacing $V$ and $H$ therein to avoid confusion):
		\begin{equation*}
			\begin{split}
				&W = H_0^1(\Omega)^d, \quad W(\mathcal{T}_h)=  H^1(\mathcal{T}_h)^d, \quad Y = [BV(\Omega)^d \cap L^4(\Omega)^d; L^4(\Omega)^d]_{1/2},	\\& \quad  X = L^2(\Omega)^d, \quad W' = H^{-1}(\Omega)^d, \quad W_h = V_h,
			\end{split}
		\end{equation*}
		where $H^1(\mathcal{T}_h)^d$ is the broken $H^1$ space equipped with the $\norm{\cdot}_{1,h}$-norm \cite{Buffa:2009,Pietro:book, Dolejsi:book},
		$BV(\Omega)^d$ is the space of functions of bounded variation \cite{Buffa:2009}, and $[Y_0, Y_1]_{\theta}$ denotes the
		complex interpolation between Banach spaces $Y_0, Y_1$ with exponent $\theta \in (0,1)$ \cite{Bergh:book}.
		%
	\end{remark}
	We present below a simple proof of a special case of \cite[Theorem 3.2]{Li:2015} 
	that stays directly within the framework of
	broken polynomial spaces and their discrete functional analysis tools. This avoids
	the need to construct a non-conforming space that embeds compactly into $L^2(\Omega)^d$ and
	is made possible by the following discrete Rellich--Kondrachov theorem
	valid for broken polynomial spaces \cite[Theorem 5.6]{Pietro:book}:
	\begin{lemma}[Discrete Rellich--Kondrachov] \label{lem:rellich-kond}
		Let $\mathcal{H} \subset (0,|\Omega|)$ be a countable set of
		mesh sizes whose unique accumulation point is $0$.  We assume
		$\cbr[0]{(\mathcal{T}_h,\mathcal{F}_h)}_{h \in \mathcal{H}}$ is a
		sequence of conforming and shape-regular simplicial meshes. Let
		$\cbr{v_h}_{h > 0}$ be a sequence in
		$\cbr[0]{V_h}_{h \in \mathcal{H}}$ bounded in the
		$\norm{\cdot}_{1,h}$-norm.  Then, for all $1 \le q < \infty$ if
		$d=2$ and $1 \le q \le 6$ if $d = 3$, the sequence
		$\cbr{v_h}_{h > 0}$ is relatively compact in $L^q(\Omega)^d$.
	\end{lemma}

	\begin{theorem}[Compactness]\label{thm:compactness}
		Let $\mathcal{H} \subset (0,|\Omega|)$ be a countable set of
		mesh sizes whose unique accumulation point is $0$. Assume that
		$\cbr[0]{(\mathfrak{T}_h,\mathfrak{F}_h)}_{h \in \mathcal{H}}$ is a
		sequence of conforming and quasi-uniform space-time prismatic
		meshes. Let $q = 4/d$. Let the sequence
		$\cbr{\boldsymbol{u}_h}_{h \in \mathcal{H}}$ be such that for each
		$h \in \mathcal{H}$,
		$\boldsymbol{u}_h \in \mathcal{V}_h^{\text{div}} \times
		\bar{\mathcal{V}}_h$. Then, $\cbr{u_h}_{h \in \mathcal{H}}$ is relatively compact in $L^2(0,T;L^2(\Omega)^d)$ if:
		\begin{enumerate}[label=(\roman*)]
			\item $\cbr[0]{\boldsymbol{u}_h}_{h \in \mathcal{H}}$ is uniformly bounded in the sense that 
			$\int_0^T \tnorm{\boldsymbol{u}_h}_{v}^2 \dif t \le M$
			for some $M>0$ independent of the mesh parameters $h$ and $\tau$. 
			\item For each $h \in \mathcal{H}$, the following bound on the discrete time derivative of $u_h$ holds uniformly for $q' = 4/(4-d)$:
			\begin{equation*}
				\bigg|\int_0^T ( \mathcal{D}_{t}^{k_t}(u_h), v_h)_{\mathcal{T}_h} \dif t \bigg| \lesssim { \del[3]{ \int_0^T \tnorm{\boldsymbol{v}_h}^{q'}_v \dif t}^{1/q'}}, \quad \forall \boldsymbol{v}_h \in \mathcal{V}_h^{\emph{div}} \times \bar{\mathcal{V}}_h.
			\end{equation*}
		\end{enumerate}
	\end{theorem}
	%
	\begin{proof}
		The proof, which proceeds in three steps, follows closely the proof
		of \cite[Theorem 3.1]{Walkington:2010} with minor modifications.
		
		\vspace{2.5mm}
		\textbf{Step one (equicontinuity):}
		Step one follows exactly the proof of \cite[Lemma 3.3]{Walkington:2010}; here
		we show that the assumptions therein can be interpreted as a uniform bound on the discrete
		time derivative \cref{eq:discrete_time_der}. By definition, it holds that
		\begin{equation}
			\int_{I_n}(\partial_t u_h, v_h )_{\mathcal{T}_h} +(\sbr{u_h}_n,v_n^+)_{\mathcal{T}_h}= \int_{I_n} (\mathcal{D}_{t}^{k_t}(u_h), v_h )_{\mathcal{T}_h} \dif t, \quad \forall v_h \in \mathcal{V}_h.
		\end{equation}
		Comparing with \cite[Lemma 3.3]{Walkington:2010}, we require $F_h: \boldsymbol{v}_h \mapsto (\mathcal{D}_t^{k_t}(u_h),v_h)_{\mathcal{T}_h}$ to be uniformly bounded $L^{q}(0,T;(V_h^{\text{div}} \times \bar{V}_h)')$, where $(V_h^{\text{div}} \times \bar{V}_h)'$ is the dual space of $V_h^{\text{div}} \times \bar{V}_h$. We show that assumption (ii) suffices. As the space $V_h^{\text{div}} \times \bar{V}_h$ and its dual are finite-dimensional (hence separable), we make the identification
		$L^q(0,T; (V_h^{\text{div}} \times \bar{V}_h)') \cong L^{q'}(0,T; V_h^{\text{div}} \times \bar{V}_h)$, 
		and we have
		\begin{equation} \label{eq:semidiscrete_dual_norm}
			\norm[0]{F_h}_{L^q(0,T; (V_h^{\text{div}} \times \bar{V}_h)')}  = \sup_{0 \ne \boldsymbol{v} \in L^{q'}(0,T; V_h^{\text{div}} \times \bar{V}_h)}   \frac{\big|\int_0^T F_h(\boldsymbol{v})\dif t\big|}{ \del{ \int_0^T \tnorm{\boldsymbol{v}}^{q'}_v \dif t}^{1/q'}}.
		\end{equation}
		Choose $F_h$ in \cref{eq:semidiscrete_dual_norm} to be the functional that maps for each $t \in [0,T]$,
		\begin{equation*}
			L^{q'}(0,T; V_h^{\text{div}} \times \bar{V}_h) \ni (v,\bar v)=: \boldsymbol{v} \mapsto  ( \mathcal{D}_{t}^{k_t}(u_h), v)_{\mathcal{T}_h} \in \mathbb{R}.
		\end{equation*}
		%
		Since $\boldsymbol{v} \in L^{q'}(0,T;V_h^{\text{div}} \times \bar{V}_h)$, we have $\Pi^t \boldsymbol{v} \in \mathcal{V}_h^{\text{div}} \times \mathcal{\bar{V}}_h$,
		and the stability of the $L^2$-projection $\Pi^t$ in $L^{q'}(I_n)$ \cite{Douglas:1975} yields
		%
		\begin{equation*}
			\begin{split}
				\norm[0]{F_h}_{L^q(0,T; (V_h^{\text{div}} \times \bar{V}_h)')} 
				& \lesssim  \sup_{\boldsymbol{v} \in L^{q'}(0,T; V_h^{\text{div}} \times \bar{V}_h)}   \frac{\big|\int_0^T ( \mathcal{D}_{t}^{k_t}(u_h), \Pi^t v )_{\mathcal{T}_h} \dif t\big|}{ \del{ \int_0^T \tnorm{\Pi^t \boldsymbol{v}}^{q'}_v \dif t}^{1/q'}} \\
				& \lesssim \sup_{\boldsymbol{v}_h \in \mathcal{V}_h^{\text{div}} \times \bar{\mathcal{V}}_h}   \frac{\big|\int_0^T ( \mathcal{D}_{t}^{k_t}(u_h), v_h)_{\mathcal{T}_h} \dif t\big|}{ \del{ \int_0^T \tnorm{\boldsymbol{v}_h}^{q'}_v \dif t}^{1/q'}},
			\end{split}
		\end{equation*}
		which is uniformly bounded by assumption (ii) of the theorem. Proceeding in an identical fashion to \cite[Lemma 3.3]{Walkington:2010}, we find that
		\begin{equation} \label{eq:equicontinuity}
			\int_{\delta}^T \norm{u_h(t) - u_h(t-\delta)}_{L^2(\Omega)}^2 \dif t 
			\le C 
			\max(\tau,\delta)^{1/q'}\delta^{1/2}.
		\end{equation}
		\vspace{2.5mm}
		
		\indent \textbf{Step two (relative compactness in
			$L^2(\theta, T-\theta; L^2(\Omega)^d)$):} We aim to show
		that for all $0 < \theta < T$, the set
		$\cbr{u_h|_{\theta,T-\theta} \; | \; h \in \mathcal{H}}$ is
		relatively compact in $L^2(\theta, T-\theta;
		L^2(\Omega)^d)$. The proof is a minor modification of
		\cite[Theorem 3.2]{Walkington:2010}.
		To this end, we construct a sequence of regularized
		functions so that we may leverage the classical
		Arzel\`{a}--Ascoli theorem (see e.g. \cite[Lemma
		1]{Simon:1986}). Let $\phi \in C^{\infty}_c(-1,1)$ be
		nonnegative with unit integral. For $\delta > 0$, set
		$\phi_\delta(s) = (1/\delta)\phi(s/\delta)$.  Extend $u_h$ by
		zero outside of $[0,T]$ and consider the sequence of mollified
		functions $\cbr{ u_h^{\delta}}_{h \in \mathcal{H}}$, where
		$u_h^{\delta}(t) = \phi_\delta * u_h(t)$.
		%
		
		Since $\int_0^T \norm{u_h}_{1,h}^2 \dif t$ is uniformly bounded by assumption, we have
		\begin{equation*}
			\begin{split}
				\norm[0]{u_h^{\delta}(t)}_{1,h}^2 
				& \le \delta \sup_{|s| < \delta} |\phi_{\delta}(t-s)|^2  \int_{0}^{\delta} \norm{u_h(s)}_{1,h}^2 \dif s
				\le M^{\star},
			\end{split}
		\end{equation*}
		with $M^{\star}$ a constant independent of $h$ and $\tau$.
		Thus, by 
		\Cref{lem:rellich-kond}, the sequence $\cbr[0]{ u_h^{\delta}(t)}_{h
			\in \mathcal{H}}$ is relatively compact in
		$L^2(\Omega)^d$ for each $t \in
		[0,T]$. Furthermore, the uniform Lipschitz continuity of the
		mollifiers $\phi_{\delta}(s)$ ensures the sequence $\cbr{
			u_h^{\delta}(t)}_{h \in \mathcal{H}}$ is equicontinuous on $[0,T]$.
		By the Arzel\`{a}--Ascoli theorem, the sequence $\cbr{ u_h^{\delta}}_{h \in \mathcal{H}}$ is relatively compact 
		in $C(0,T;L^2(\Omega)^d)$ and thus also $L^2(0,T;L^2(\Omega)^d)$ as the former embeds continuously into the latter.
		As relatively compact sets are totally bounded,
		\begin{equation*} \label{eq:lp_totally_bnd}
			\forall \epsilon > 0,\; \exists {h_1,\dots, h_M} \subset \mathcal{H} \; s.t. \; \cbr[0]{ u_h^{\delta}}_{h \in \mathcal{H}}  \subset \bigcup_{i=1}^M B_{\epsilon}(u_{h_i}),
		\end{equation*}
		where $B_{\epsilon}$ is an $\epsilon$-ball  in the metric induced by the $L^2(0,T;L^2(\Omega)^d)$-norm.
		The remainder of the proof that the set $\cbr{u_h|_{[\theta,T-\theta]} \; | \; h \in \mathcal{H}}$ is relatively compact in $L^2(\theta,T-\theta;L^2(\Omega)^d)$ for all $0 < \theta < T/2$ is identical to that of \cite[Theorem 3.2]{Walkington:2010}.
		\vspace{2.5mm}
		
		\textbf{Step three (finishing up):}
		The equicontinuity \cref{eq:equicontinuity} and \cite[Lemma 3.4]{Walkington:2010} ensure that
		the sequence $\cbr{u_h}_{h \in \mathcal{H}}$ is bounded uniformly in $L^r(0,T;L^2(\Omega)^d)$ for $1 \le r <4$.
		Consequently, for all $\epsilon >0$ we can find $\theta > 0$ such that
		\begin{equation*} \label{eq:uniform_integr}
			\int_0^{\theta} \norm{u_h(t)}_{L^2(\Omega)}^2 \dif t +	\int_{T-\theta}^{T} \norm{u_h(t)}_{L^2(\Omega)}^2 \dif t \le \epsilon.
		\end{equation*}
		%
		It follows that $\cbr{u_h \; | \; h \in \mathcal{H}}$ is the uniform limit of relatively compact sets in
		$L^2(0,T;L^2(\Omega)^d)$ \cite[Section 2]{Simon:1986}.
		%
		Thus, $\cbr{u_h \; | \; h \in \mathcal{H}}$ is relatively compact in $L^2(0,T;L^2(\Omega)^d)$.
	\end{proof}

	\section{Properties of the projections $\Pi$ and $\bar{\Pi}$}
	
	\subsection{Approximation properties of $\Pi^t$ and $\Pi_h^{\text{div}}$}
	\begin{lemma}[Approximation properties of $\Pi_h^{\text{div}}$ and $\Pi^t$]  \label{lem:approximation_div}
		Let $\ell \ge 0$ and suppose that $\eta \in W^{\ell+1,\infty}(0,T)$ and $\psi \in H^{\ell+1}(\Omega)^d$. Then, for all $n = 0,\dots,N-1$,
		\begin{equation} \label{eq:l2_time_proj_est}
			\norm[0]{u - \Pi^t u}_{L^{\infty}(I_n) } \lesssim \tau^{\ell+1} |\eta|_{W^{\ell+ 1,\infty}(I_n)},
		\end{equation}
		and if $\mathcal{T}_h$ is conforming and quasi-uniform, we have for $0 \le m \le 2$ and $K \in \mathcal{T}_h$,
		\begin{align} \label{eq:div_proj_est_1}
			\sum_{K \in \mathcal{T}_h}\norm[0]{\psi - \Pi_h^{\emph{div}} \psi}_{H^{m}(K)}^2 &\lesssim h^{2(\ell - m + 1)} |\psi|_{H^{\ell+1}(\Omega)}^2, \\
			\norm[0]{\psi - \Pi_h^{\emph{div}} \psi}_{L^{\infty}(\Omega)} &\lesssim h^{1/2} |\psi|_{H^{2}(\Omega)}, \label{eq:div_proj_est_2}
		\end{align}
		the latter requiring $\ell \ge 1$.
	\end{lemma}
	\begin{proof}
		Estimate \cref{eq:l2_time_proj_est} is standard, see e.g. \cite[Lemma 11.18]{Ern:booki}.
		The proof of \cref{eq:div_proj_est_1} can be found in \cite{Kirk:2021}; we note that therein it is assumed
		that $\ell \ge 1$ but the proof easily extends to the case $\ell = 0$.
		We now show \cref{eq:div_proj_est_2}. 
		Let $\widehat{K}$ be the reference simplex in $\mathbb{R}^d$ and suppose that $F_K: \widehat{K} \to K$ is an affine mapping;
		denote its Jacobian matrix by $J_K$.
		As $H^2(\widehat K) \subset L^{\infty}(\widehat K)$ with continuous embedding for $d \le 3$, we have
		by repeated use of \cite[Lemma 11.7]{Ern:booki}:
		\begin{equation*} 
			\begin{split}
				\norm[0]{\psi - \Pi_h^{\text{div}} \psi}_{L^{\infty}(K)} 
				& \lesssim \norm[0]{J_K}_{\ell^2}^2 |\det J_K|^{-1/2} \norm[0]{\psi - \Pi_h^{\text{div}} \psi}_{H^{2}(K)}.
			\end{split}
		\end{equation*}
		Since $\mathcal{T}_h$ is assumed quasi-uniform and hence shape-regular, we have 
		$\norm[0]{J_K}_{\ell^2}^2 \lesssim h_K^2$ and $|\det J_K|^{-1/2} \lesssim h_K^{-d/2}$ (see e.g. \cite[Lemma 11.1]{Ern:booki}, \cite[Chapter 1.2]{Sayas:book}).
		Thus, for $d \le 3$,
		\begin{equation*}
			\begin{split}
				\norm[0]{\psi - \Pi_h^{\text{div}} \psi}_{L^{\infty}(K)} 
				& \lesssim h^{1/2} \norm[0]{\psi}_{H^{2}(\Omega)}.
			\end{split}
		\end{equation*}
		The result follows by noting that this bound holds uniformly for all $K \in \mathcal{T}_h$.
	\end{proof}

	\subsection{Proof of \Cref{prop:proj_bnd}}
	\label{ss:proj_bnd_proof}
	
	It suffices to show the inequality in \cref{eq:proj_bnd} on a single
	space-time slab $\mathcal{E}^n$; the result then follows
	by summing over all space-time slabs. Let $\varphi \in \mathscr{M}$.
	By the definitions of the norm $\tnorm{\cdot}_v$ and the projections $\Pi \varphi$ and $\bar{\Pi} \varphi$ given in \cref{eq:test_func_for_time_der_bnd}, we have
	\begin{equation*}
		\begin{split}
			&\tnorm{(\Pi \varphi, \bar{\Pi} \varphi)}_v^{4/(4-d)}  \\
			& =  \del[3]{ \sum_{K \in \mathcal{T}_h} \int_{K} |\nabla \Pi_h^{\text{div}} \Pi^t \sum_{k=1}^M\eta_k \psi_k|^2 \dif x +  \sum_{K \in \mathcal{T}_h}  h_K^{-1} \int_{\partial K} |(\Pi_h^{\text{div}} - \bar{\Pi}_h) \Pi^t \sum_{k=1}^M \eta_k\psi_k|^2 \dif x }^{2/(4-d)}.
		\end{split}
	\end{equation*}
	Available approximation results for the projection $\bar{\Pi}_h$ and \cref{eq:div_proj_est_1} 
	yield for $\Psi \in H^1(\Omega)^d$,
	\begin{equation*}
		\sum_{K \in \mathcal{T}_h} \del{ \norm[0]{\nabla \Pi_h^{\text{div}} \Psi}_{L^2(K)}^2 + h_K^{-1} \norm[0]{(\Pi_h^{\text{div}}- \bar{\Pi}_h)\Psi}_{L^2(\partial K)}^2 } \lesssim \norm{\nabla \Psi}_{L^2(\Omega)}^2.
	\end{equation*}
	Therefore, we have
	\begin{equation} \label{eq:proj_bnd2_inter}
		\tnorm{(\Pi \varphi, \bar{\Pi} \varphi)}_v^{4/(4-d)} \lesssim
		\del[4]{\int_{\Omega} |\Pi^t \sum_{k=1}^M\eta_k \nabla  \psi_k|^2 \dif x }^{2/(4-d)}, \quad \forall \varphi \in \mathscr{M}.
	\end{equation}
	If $d=2$, we can integrate \cref{eq:proj_bnd2_inter} over $I_n$ and use Fubini's
	theorem and the stability of the projection $\Pi^t$ in $L^2(I_n)$ to find
	\begin{equation*}
		\int_{I_n} \tnorm{(\Pi \varphi, \bar{\Pi} \varphi)}_v^{2}\dif t  
		\lesssim \int_{I_n} \norm{\varphi}_{V}^2 \dif t, \quad \forall \varphi \in \mathscr{M},
	\end{equation*}
	as required.
	On the other hand, if $d=3$, we integrate \cref{eq:proj_bnd2_inter} over $I_n$,
	and apply a finite-dimensional scaling argument between norms in $L^2(I_n)$ and $L^1(I_n)$
	(see e.g. \cite[Lemma 1.50]{Pietro:book}) to find:
	\begin{equation} \label{eq:proj_bnd_3d_inter}
		\int_{I_n} \tnorm{(\Pi \varphi , \bar{\Pi} \varphi )}_v^4 \dif t 
		\lesssim \tau^{-1}  \del[4]{\int_{I_n} \int_{\Omega} |\Pi^t \sum_{k=1}^M\eta_k \nabla  \psi_k|^2 \dif x \dif t }^{2}.
	\end{equation}
	Using Fubini's theorem to interchange the temporal and spatial integrals in \cref{eq:proj_bnd_3d_inter} as necessary, we can apply
	the stability of the projection $\Pi^t$ in the $L^2(I_n)$ norm followed by the Cauchy--Schwarz inequality applied
	to the temporal integral to find
	\begin{equation*}
		\begin{split}
			\int_{I_n} \tnorm{(\Pi \varphi , \bar{\Pi} \varphi )}_v^4 \dif t  & \lesssim \tau^{-1}   \del[3]{ \int_{I_n} \norm{\varphi}_V^2 \dif t }^{2}  \lesssim \int_{I_n} \norm{\varphi}_V^4 \dif t, \quad \forall \varphi \in \mathscr{M}.
		\end{split}
	\end{equation*}

\bibliographystyle{amsplain.bst}
\bibliography{references}
\end{document}